\documentclass[10pt]{article}
\usepackage[a4paper]{geometry}
\usepackage{amsmath, amssymb, amsthm}

\usepackage{appendix}
\usepackage{multicol} 
\usepackage{makeidx}
\usepackage{cite}
\usepackage{captdef}
\usepackage{paralist}
\usepackage{ dsfont }

\usepackage[colorlinks=true]{hyperref}
\hypersetup{urlcolor=blue, citecolor=black}

\usepackage{psfrag}
\usepackage{graphicx}
\usepackage{graphics}
\usepackage[dvipsnames]{xcolor}
\usepackage{comment}
\usepackage[showonlyrefs]{mathtools}
\mathtoolsset{showonlyrefs=true}

\usepackage{fullpage}

\usepackage{subcaption}

    \newtheorem{theo}{Theorem}[section]
    \newtheorem{coro}[theo]{Corollary}
    \newtheorem{lemma}[theo]{Lemma}
    \newtheorem{prop}[theo]{Proposition}
    \theoremstyle{definition}
    
    \newtheorem{remark}[theo]{Remark}

\newcommand{\R}{\mathbb{R}}

\newcommand{\pfrak}{\mathfrak{p}}
\newcommand{\qfrak}{\mathfrak{q}}

\newcommand{\eps}{\varepsilon}

\newcommand{\abs}[1]{\left|#1\right|}
\newcommand{\norm}[1]{\left \| #1\right \|}

\DeclareMathOperator{\sign}{sign}

\newcommand{\N}{{\mathbb{N}}}

\renewcommand{\epsilon}{\varepsilon}

\numberwithin{equation}{section}

\def\sideremark#1{\ifvmode\leavevmode\fi\vadjust{\vbox to0pt{\vss
 \hbox to 0pt{\hskip\hsize\hskip1em
 \vbox{\hsize2.1cm\tiny\raggedright\pretolerance10000
  \noindent #1\hfill}\hss}\vbox to15pt{\vfil}\vss}}}%

\title{On least energy solutions to a pure Neumann Lane-Emden system: convergence, symmetry breaking, and multiplicity}
\author{Alberto Saldaña, Delia Schiera, and Hugo Tavares}

\date{}

\begin{document}
\maketitle

\begin{abstract}
We consider the following Lane-Emden system with Neumann boundary conditions
\[
-\Delta u= |v|^{q-1}v  \text{ in } \Omega,\qquad
-\Delta v= |u|^{p-1}u  \text{ in } \Omega,\qquad
\partial_\nu u=\partial_\nu v=0  \text{ on } \partial \Omega, 
\]
where $\Omega$ is a bounded smooth domain of $\R^N$ with $N \ge 1$. We study the multiplicity of solutions and the convergence of least energy (nodal) solutions (l.e.s.) as the exponents $p, q > 0$ vary in the subcritical regime $1/(p+1) + 1/(q+1)  > (N-2)/N$, or in the critical case $1/(p+1) + 1/(q+1) =(N-2)/N$ with some additional assumptions. We consider, for the first time in this setting, the cases where one or two exponents tend to zero, proving that l.e.s. converge to a problem with a sign nonlinearity. 
Our approach is based on an alternative characterization of least energy levels in terms of the nonlinear eigenvalue problem
\[
\Delta (|\Delta u|^{\frac 1 q -1} \Delta u) = \lambda |u|^{p-1} u, \quad \partial_\nu u=\partial_\nu(|\Delta u|^{\frac 1 q -1} \Delta u)=0 \text{ on } \partial \Omega.
\]

As an application, we show a symmetry breaking phenomenon for l.e.s. of a bilaplacian equation with sign nonlinearity and for other equations with nonlinear higher-order operators.

\medskip

\noindent \textbf{Mathematics Subject Classification:} 35J50, 35J47, 35B40 (Primary), 35B07, 35B33, 35B38, 35J30, 35P30. 

\medskip

\noindent \textbf{Keywords: } Neumann boundary conditions, asymptotic analysis, dual method, symmetry breaking, biharmonic operator, sign nonlinearity.  

\end{abstract}

\section{Introduction}

Let $\Omega$ be a bounded smooth domain in $\R^N$, $N\geq 1$, and consider the following pure Neumann Lane-Emden system
\begin{equation}\label{system} 
-\Delta u= |v|^{q-1}v  \text{ in } \Omega,\qquad
-\Delta v= |u|^{p-1}u  \text{ in } \Omega,\qquad
\partial_\nu u=\partial_\nu v=0  \text{ on } \partial \Omega,
\end{equation}
where $p>0$ and $q>0$ are some suitable exponents. These problems are characterized by the fact that all the nontrivial solutions change sign. Indeed, observe that if $u$ and $v$ solve \eqref{system}, then, by the divergence theorem,
\[
\int_\Omega |u|^{p-1}u=\int_\Omega |v|^{q-1}v=0.
\]
System \eqref{system} has a variational structure and least energy solutions can be found by constrained minimization under suitable assumptions on the exponents. For instance, if $pq\neq 1$ and $p$ and $q$ are below the critical hyperbola, i.e.
\begin{equation}\label{CH} 
p,q>0,\qquad \frac{1}{p+1} +\frac{1}{q+1}  >  \frac{N-2}{N},
\end{equation}
then existence (and some qualitative properties) of least energy solutions to \eqref{system} is shown in \cite{ST2}. Furthermore, in \cite{PST} it is proved that system \eqref{system} has least energy solutions when $N\geq 4$ and the exponents are on the critical hyperbola, at least under some additional restrictions. This is in sharp contrast with respect to its Dirichlet counterpart, for which least energy solutions never exist in the critical case (see, for instance, \cite[Lemma 2.1]{g08} and \cite{w93,cs20}). These additional restrictions are needed to use a compactness argument based on Cherrier-type inequalities (following \cite{c91}). To be more precise, in the critical regime one needs that $N\geq 4$ and
\begin{equation}\label{CH2}
\begin{cases}
p,q>0,\qquad \frac{1}{p+1} +\frac{1}{q+1}  =  \frac{N-2}{N},\quad \\[10pt]
N \geq 6\text{ and } p, q>\frac{N+2}{2(N-2)},\quad \text{ or } N=5 \text{ and } p, q>\frac{17}{13},\quad \text{ or } N=4 \text{ and } p, q>\frac{7}{3}.
\end{cases}
\end{equation}

A natural question is whether the least energy subcritical solutions tend to the least energy critical ones as the exponents approach the hyperbola. 

In this paper, we answer this question and we characterize several other interesting limit behaviors such as the asymptotically linear case ($pq=1$) or the asymptotically non-smooth case
\begin{equation}\label{p=0}
p=0,\qquad 0<q<\infty
\end{equation}
(the case $q=0$ with $0<p<\infty$ follows by symmetry). We also treat here the limiting case $p=q=0$, with a different approach. To be more precise, let
\begin{align*}
\mathcal{A}:=\{(p,q)\in \R^2 \text{ satisfying either } \eqref{CH}, \eqref{CH2}, \text{ or } \eqref{p=0}\}
\quad \text{ and } \quad \mathcal{H}:=\{(p,q)\in \R^2:\ p,q>0,  \ pq=1\}.
\end{align*}

\begin{figure}[h!]
    \centering
\begin{picture}(150,130) 
    \includegraphics[width=0.3\linewidth]{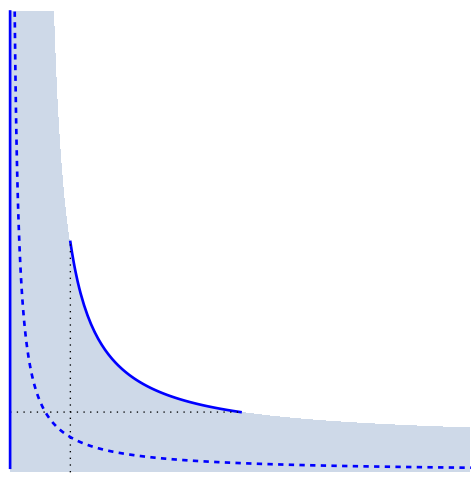}
  \put(3,0){$p$} 
  \put(-145,130){$q$}
\end{picture}
    \caption{The shaded region together with the dark curves are the admissible set of exponents $\mathcal A$.  The dashed hyperbola represents $\mathcal H$. The small dotted lines are the restrictions in \eqref{CH2}, whenever $N\geq 4$.}
    \label{fig:i}
\end{figure}

If $p,q \ne 0$, we say that a strong solution of \eqref{system} is a pair $(u,v)\in W^{2,\frac{q+1}{q}}(\Omega)\times W^{2,\frac{p+1}{p}}(\Omega)$ that satisfies the equations \eqref{system} $a.e.$ in $\Omega$ and the boundary conditions in the trace sense. Let 
\begin{align*}
    I_{p,q}(u,v):=\int_\Omega \nabla u \nabla v - \frac{1}{p+1}\int_\Omega |u|^{p+1}-\frac{1}{q+1}\int_\Omega |v|^{q+1}\, dx
\end{align*}
be the associated Euler-Lagrange functional of \eqref{system}.  Since 
\begin{align}\label{em}
W^{2,\frac{q+1}{q}}(\Omega)\times W^{2,\frac{p+1}{p}}(\Omega) \hookrightarrow L^{p+1}(\Omega)\times L^{q+1}(\Omega)\qquad \text{ for }(p,q)\in \mathcal{A},
\end{align}
 we can define the least energy level of \eqref{system} as
\[
c_{p,q}:=\inf\{I_{p,q}(u,v):\ (u,v) \text{ is a strong solution of } \eqref{system}\},\qquad (p,q)\in \mathcal{A}\setminus \mathcal{H}.
\]
Minimizers of $c_{p,q}$ are called \emph{least energy solutions}. 

If \eqref{p=0} holds, then we interpret \eqref{system} as
\begin{equation}\label{system:p=0} 
-\Delta u= |v|^{q-1}v  \text{ in } \Omega,\qquad
-\Delta v= \sign(u)  \text{ in } \Omega,\qquad
\partial_\nu u=\partial_\nu v=0  \text{ on } \partial \Omega.
\end{equation}
Here, the function $\sign$ is defined by
\begin{align}\label{sgn:def}
\sign(u) = \begin{cases}
1 & \text{if } u > 0, \\
-1 & \text{if } u < 0, \\
0 & \text{if } u = 0.
\end{cases}
\end{align}
In this setting, a strong solution of \eqref{system:p=0} is a pair $(u,v)$ which belongs to $W^{2,\frac{q+1}{q}}(\Omega)\times W^{2,r}(\Omega)$ for every $r\geq 1$ and satisfies the equations in \eqref{system:p=0} $a.e.$ in $\Omega$ and the boundary conditions in the trace sense. Similarly as before, strong solutions satisfy that
\[
|\{u>0\}|-|\{u<0\}|=\int_\Omega |v|^{q-1}v=0,
\]
which implies that all nontrivial solutions of \eqref{system:p=0} must change sign. One of the main difficulties in this setting is that the functional $I_{0, q}$ is \textit{not} differentiable and the nonlinearities are not invertible (which prevents the dual approach). We overcome these difficulties by extending some arguments from \cite{PW} and using some tools from distribution theory.

\paragraph{Main results: asymptotic study in the exponents.} To state our first main result, we need to define the first (nontrivial) pure Neumann eigenvalue for \eqref{system}. For $q>0$, let
\begin{equation}\label{mu1} 
\mu_{1,q}:= \inf \left\{ \frac{ \|\Delta u\|_\frac{q+1}{q}}{\|u\|_\frac{q+1}{q}}, u \in W^{2, \frac{q+1}{q}}_\nu(\Omega) \setminus \{ 0 \}: \int_\Omega |u|^{\frac 1 q -1} u =0\right\},
\end{equation} 
where $W^{2, \frac{q+1}{q}}_\nu(\Omega):=\{u\in W^{2, \frac{q+1}{q}}(\Omega):\ \partial_\nu u=0 \text{ on } \partial \Omega\}.$ Note that this is a closed subspace of $W^{2,\frac{q+1}{q}}(\Omega)$ endowed with the norm  $u\mapsto \|u\|_1+\|\Delta u\|_\frac{q+1}{q}$ (see Corollary~\ref{Coro:equivalentnorms} below).  A standard argument using the compact embedding $W^{2,\frac{q+1}{q}}(\Omega)\hookrightarrow L^\frac{q+1}{q}(\Omega)$ shows that $\mu_{1,q}$ is achieved.  The eigenvalue $\mu_{1,q}$ generates at least one (depending on its multiplicity) curve of eigenvalues with corresponding eigenfunctions, see, for instance, \cite[Section 2.4]{dSNST}.  

The following is our main existence and convergence result, which characterizes the properties of least energy solutions with respect to variations in the exponents $p$ and $q$ in the set $\mathcal{A}$.

\begin{theo}\label{prop:convergence direct}
For every $(p,q)\in \mathcal{A}\setminus \mathcal{H}$, the level $c_{p,q}$ is achieved, and the map 
\[
\mathcal{A}\setminus \mathcal{H}\to \R,\qquad  (p, q) \mapsto c_{p, q} \]
is continuous. 
Moreover, let $(p_n,q_n)\in \mathcal{A}\setminus \mathcal{H}$ and $(p,q)\in \mathcal{A}$ be such that $(p_n, q_n) \to (p, q)$. Let $(u_n, v_n)$ be a corresponding least energy solution. 
\begin{itemize}
\item[(i)] If $(p,q) \notin \mathcal{H}$, then up to a subsequence there exists a least energy solution $(u, v)$ such that if $p>0$, 
\begin{equation}\label{conv} (u_n, v_n) \to (u, v) \quad \text{ in } C^{2, \zeta}(\overline \Omega) \times C^{2, \eta}(\overline \Omega), \quad \text{for any $\zeta \in (0, \min\{ q , 1\})$ and $\eta \in (0, \min\{ p , 1\})$.} \end{equation}
 In case $p=0$, then 
\[ (u_n, v_n) \to (u, v) \quad \text{ in } C^{2, \zeta}(\overline \Omega) \times C^{1, \eta} (\overline \Omega), \quad \text{for any $\zeta \in (0, \min\{ q , 1\})$ and $\eta \in (0, 1)$.}\]
 
\item[(ii)] Let $(p,q) \in \mathcal{H}$ and $\mu_{1,q}>1$. If $p_n q_n \to 1^+$, then 
\[ \norm{u_n}_{L^\infty} \to \infty,  \quad \norm{v_n}_{L^\infty} \to \infty, \quad c_{p_n, q_n} \to \infty. \]
If $p_n q_n \to 1^-$, then 
\[ \norm{u_n}_{L^\infty} \to 0,  \quad \norm{v_n}_{L^\infty} \to 0, \quad c_{p_n, q_n} \to 0. \]
\item[(iii)] Let $(p,q) \in \mathcal{H}$ and $\mu_{1,q} <1$. If $p_n q_n \to 1^+$, then 
\[ \norm{u_n}_{L^\infty} \to 0,  \quad \norm{v_n}_{L^\infty} \to 0, \quad c_{p_n, q_n} \to 0. \]
If $p_n q_n \to 1^-$, then 
\[ \norm{u_n}_{L^\infty} \to \infty,  \quad \norm{v_n}_{L^\infty} \to \infty, \quad c_{p_n, q_n} \to \infty. \]
\end{itemize}
\end{theo}
The proof of Theorem~\ref{prop:convergence direct} relies on the dual method, the reduction-by-inversion approach (which relates \eqref{system} with a higher-order quasilinear equation, see \cite{BdST}), and the use of auxiliary nonlinear eigenvalue problems; for instance, for $(p,q)\in \mathcal{A}$ with $p>0$, we consider 
\begin{align}\label{lambdapq}
\Lambda_{p,q}
&:=\inf \left \{ \|\Delta u\|_{\frac{q+1}{q}} : u \in W^{2, \frac{q+1}{q}}_\nu(\Omega),\ \norm{u}_{p+1}=1,\ \int_\Omega |u|^{p-1} u=0 \right\}
\end{align}
(observe that $\mu_{1,q}=\Lambda_{\frac{1}{q},q}$). Minimizers of $\Lambda_{p,q}$ are in one-to-one correspondence with solutions of \eqref{system} (see Lemma~\ref{lemma lambda c}), but the additional normalization condition turns out to be quite helpful when considering convergence properties.  Theorem~\ref{prop:convergence direct} follows from the analogous result for minimizers of $\Lambda_{p,q}$, which  is a result of independent interest and can be found in Theorem~\ref{main thm lambda} below.

We point out that the additional assumptions \eqref{CH2} in the critical case arise only in the proof of a compactness condition, and are most likely not optimal, see Remark~\ref{rmk:conditions}.

The case $(p,q)\in \mathcal{H}$ with $\mu_{1,q}=1$ is more delicate since the way in which $p_n$ and $q_n$ approach their limits plays a role.  The following result shows that, under suitable conditions, least energy solutions converge to eigenfunctions as the exponents approach any point at the hyperbola $\mathcal H.$

\begin{theo}\label{mu1:prop}
Let $\pfrak,\qfrak\in C^1([0,1])$ be such that $\pfrak(0)\qfrak(0)=1$. Let $e(t):=\pfrak(t)\qfrak(t)-1$ and, for $n\in \mathbb N$, let
\begin{align*}
p_n:=\pfrak(\tfrac{1}{n}),\quad q_n:=\qfrak(\tfrac{1}{n}),\quad p:=\pfrak(0),\quad q:=\qfrak(0).
\end{align*}
Assume that $\mu_{1,q}=1$ and let $(u_n, v_n)$ be a least energy solution achieving $c_{p_n, q_n}$ for $n\in \mathbb N$. If $e'(0)\neq 0$, then, up to a subsequence, there exists a couple $(\varphi_1, \psi_1)$ which solves
\begin{equation}\label{eigen}
-\Delta \varphi_1 = \mu_{1,q} |\psi_1|^{q-1} \psi_1, \quad -\Delta \psi_1 = \mu_{1,q} |\varphi_1|^{\frac 1 q-1} \varphi_1\quad \text{ in }\Omega, \qquad \partial_\nu \varphi_1=\partial_\nu \psi_1=0 \text{ on } \partial \Omega,
\end{equation}
together with
\begin{equation}\label{eigen normaliz} \int_\Omega |\varphi_1|^{p-1}\varphi_1=\int_\Omega |\psi_1|^{q-1}\psi_1=0, \quad \norm{\varphi_1}_{p+1}=\norm{\psi_1}_{q+1}=1, \end{equation}
and such that $(u_n, v_n) \to (\mathfrak{c}^\frac{1}{p+1}\varphi_1, \mathfrak{c}^\frac{1}{q+1} \psi_1)$ in $L^{p+1}(\Omega) \times L^{q+1}(\Omega)$, 
where
\[ \mathfrak{c} := \lim_{n\to \infty} \Lambda_{p_n, q_n}^{\frac{(p_n+1)(q_n+1)}{p_n q_n-1}}\in(0,\infty).\]
Moreover, up to a subsequence, $\lim_{n\to \infty}\frac{c_{p_n, q_n}}{p_nq_n-1}=\frac{1}{4} \mathfrak{c}.$
\end{theo}

If $p=q=1$, the constant $\mathfrak{c}$ can be obtained in terms of the eigenfunctions $(\varphi_1, \psi_1)$, see Remark~\ref{formula}.  About the assumption $e'(0) \ne 0$, see Remark \ref{rem_hoe}. Note also that, if $p=q\geq 0$, then \eqref{system} reduces to the scalar equation
\begin{align}\label{se}
-\Delta u = |u|^{p-1} u \text{ in } \Omega, \quad \partial_\nu u=0 \text{ on } \partial \Omega, 
\end{align}
see \cite[Lemma 2.6]{ST2} or Lemma~\ref{p=q} below. For the single equation \eqref{se}, the convergence results in Theorem~\ref{prop:convergence direct} and~\ref{mu1:prop} are established in \cite{ST}. We also refer to \cite{bbgv08} and \cite{g09} for some related results for scalar equations with Dirichlet boundary conditions. 

\smallbreak

Next, we switch our attention to the case when both exponents go to zero. In this case, the limit problem is again a scalar Neumann equation with sign nonlinearity, which is studied in \cite{PW}. To state our result, let 
\begin{align}\label{i0}
I_0(u):=\frac 12 \int_\Omega |\nabla u|^2 - \int_\Omega |u|,\qquad c_0:= \inf\{I_0(u): u\in \mathcal{M} \},
\end{align}
where $\mathcal{M}:=\left\{u\in H^1(\Omega): \Big| |\{ u >0\}| - |\{ u < 0 \}| \Big | \le |\{ u=0 \} |\right\}.$ By \cite[Theorem 1.2]{PW}, the level $c_0$ is attained, and every minimizer is a sign changing solution of 
\begin{equation}\label{eq:p=q=0} 
- \Delta u = \sign(u) \quad \text{ in } \Omega, \qquad \partial_\nu u=0 \quad \text{ on } \partial \Omega. 
\end{equation}

We can also characterize the limit when \emph{both} the exponents go to zero. 
\begin{theo}\label{p, q 0}
   If $(p_n, q_n) \to (0,0)$, then $c_{p_n, q_n} \to 2 c_0$. Moreover, if $(u_n, v_n)$ is an optimizer of $c_{p_n, q_n}$, then (up to a subsequence) $u_n,v_n \to u$ in $C^{1, \zeta}(\overline \Omega)$ for any $\zeta \in (0,1)$, where $u$ is an optimizer for $c_0.$ 
\end{theo}

\paragraph{Main results: symmetry breaking.} Convergence results can be useful to characterize qualitatively limit cases that are difficult to study with other techniques. To illustrate this point, we use Theorem~\ref{prop:convergence direct} to show a new symmetry breaking result for a nonlinear equation involving the bilaplacian and a sign nonlinearity. 

To be more specific, let $B:=\{x\in\R^N\::\: |x|<1\}$. For $(p,q)\in \mathcal{A}^+:=\{(p,q)\in \mathcal{A}: p>0\}$, it is known that least energy solutions of \eqref{system} are not radially symmetric if $\Omega=B$, see \cite[Theorem 1.1]{ST} for the subcritical case and \cite[Theorem 1.9]{PST} for the critical one. These symmetry breaking results rely on a flip-\&-rearrange transformation that can only be used in a dual setting. Therefore, this approach can not be used when considering the $\sign$ nonlinearity.  For $(p,q)=(0,0)$, as mentioned before, this corresponds to the single equation \eqref{eq:p=q=0} and symmetry breaking of least energy solutions is shown in \cite[Theorem 1.2]{PW}, using unique continuation properties that are not known in the higher-order case. 

In our next result, we combine Theorem~\ref{prop:convergence direct} together with results from \cite{ST} and \cite{PW} to show a symmetry breaking result for least energy solutions of
\begin{equation}\label{bilap:sb}
\Delta^2 u = \sign (u)\quad \text{ in  } B,\qquad \partial_\nu u= \partial_\nu  \Delta u=0\quad \text{ on } \partial B.
\end{equation}
Note that \eqref{bilap:sb} is a particular case of \eqref{system:p=0} for $q=1$. In particular, 
\[
c_{0,1}=m:=\inf_{\mathcal{M}_1}\varphi,\qquad \text{ where }\varphi(w) := \frac{1}{2} \int_\Omega |\Delta w |^{2} - \int_\Omega |w|
\]
and  
\[
\mathcal{M}_1:=\left\{ u \in W_\nu^{2, 2}(\Omega): \Big| |\{ u >0\}| - |\{ u < 0 \}| \Big | \le |\{ u=0 \} | \right \}. 
\]
A function achieving $m$ is called a least energy solution of \eqref{bilap:sb}.
\begin{theo}\label{symm:break:thm}
For $N\geq 2$, least energy solutions of the biharmonic problem \eqref{bilap:sb} are not radially symmetric.
\end{theo}

 The bilaplacian case is much more involved than the Laplacian case (treated in \cite[Theorem 1.2]{PW}), due to the higher-order nature of the equation and because some important tools (such a unique continuation results) are not yet available.  To show Theorem \ref{symm:break:thm} we follow some ideas from \cite{PW}.  We prove that the (global) least energy level is strictly below the least energy level among radially symmetric functions. For this, we use Theorem~\ref{prop:convergence direct} and previously known results for Neumann Hamiltonian systems \cite{ST2} to find an explicit formula for the solution of \eqref{bilap:sb} in the radial setting (see formulas \eqref{form:u:1}, \eqref{form:u:2}, \eqref{form:u:3}, and \eqref{form:u:4}) and we compute its energy (see \eqref{ene1} and Table~\ref{tab1} in Section~\ref{sec:symm:b}). Then, we construct suitable \emph{nonradial} competitor with lower energy, which immediately yields the symmetry breaking result.

Note that, if the boundary conditions in \eqref{bilap:sb} are replaced by Dirichlet ($u= \partial_\nu u=0$ on $\partial B$) or Navier $(u= \Delta u=0$ on $\partial B$) boundary conditions, then, by the method of decomposition with respect to dual cones (see \cite[Theorem 7.19]{GGS}), the corresponding least energy solutions are nonnegative and therefore they would be (unique and radial) torsion functions. In these cases, one would need to consider least energy \emph{nodal} solutions to see the effects of the nonlinearity. 

We also mention that the symmetry breaking is expected to hold for \eqref{system} in the whole range \eqref{p=0}. Unfortunately, our strategy requires some explicit computations, which are difficult to obtain if $q\neq 1$. However, combining Theorem~\ref{symm:break:thm}, the known symmetry breaking results for \eqref{eq:p=q=0}, and the convergence results from Theorem~\ref{prop:convergence direct} and Theorem~\ref{p, q 0}, we have the following direct consequence (where $a$ plays the role of the exponent $1/q-1$).
\begin{coro}\label{coro:intro}
There exists $\varepsilon>0$ such that, if $(u,v)$ is a least energy solution of \eqref{system} with $p=0$ and $q\in (0,\eps)\cup (1-\eps,1+\eps)$, then $u,v$ are not radially symmetric. In particular, for $a$ close to $0$ or $a$ sufficiently large, least energy solutions of
\[
\Delta (|\Delta u|^a \Delta u)=\sign(u)\quad \text{ in  } B,\qquad \partial_\nu u= \partial_\nu  (|\Delta u|^a\Delta u)=0\quad \text{ on } \partial B,
\]
are not radially symmetric. 
\end{coro}

\paragraph{Main results: multiplicity.} In the previous results, we have focused on least energy solutions. However, we remark that \eqref{system} has infinitely many solutions. Indeed, using Ljusternik-Schnirelmann theory and the dual method we can show the following. 
\begin{theo}\label{prop:mult}
Let $pq\ne 1$ be such that $(p, q)$ satisfies \eqref{CH}, that is, it is below the critical hyperbola. Then, there exist infinitely many solutions to~\eqref{system}.
\end{theo}
We point out that, for $p=q$, Theorem~\ref{prop:mult} also gives a different and simpler proof of \cite[Theorem 2.1]{MiaoDu} in the scalar case. We prove this result using a dual approach, which allows to easily overcome the possible issues arising from the fact that $p$ or $q$ may be smaller than one.

To close this introduction, we also mention that, using the Lyapunov-Schmidt reduction method, it is possible to construct solutions of \eqref{system} in the slightly supercritical regime, namely, \emph{above} the critical hyperbola, we refer to \cite{PST23} for the scalar equation case and to \cite{GP25} for the case of Hamiltonian systems.

\medskip 

The paper is organized as follows. In Sections \ref{sec:ne} and \ref{sec:conv1} we state our main existence and convergence results about the auxiliary nonlinear eigenvalue problems $\Lambda_{p,q}$ introduced in \eqref{lambdapq} and \eqref{Lambda0q}, which are of independent interest. Then, in Section \ref{sec:conv:2}, we use these results to show Theorems \ref{prop:convergence direct}, \ref{mu1:prop}, and \ref{p, q 0}.  The symmetry breaking results Theorem \ref{symm:break:thm} and Corollary \ref{coro:intro} are shown in Section \ref{sec:symm:b}. Finally, the multiplicity result is proved in Section~\ref{sec:mul}.

\medskip

\noindent \textbf{Notation. } For $s>1$, we denote the $L^s(\Omega)$ and $W^{2,s}(\Omega)$ norms by $\|\cdot\|_s$ and $\|\cdot \|_{W^{2,s}}$ respectively.  We use $\norm{\cdot}_{L^\infty}$ for the usual supremum norm.  Whenever the integration domain is $\R^N$, we write it explicitly and denote the $L^2(\R^N)$ and $W^{2,s}(\R^N)$ norms by $\|\cdot \|_{L^s(\R^N)}$ and $\|\cdot \|_{W^{2,s}(\R^N)}$ respectively. We denote $\mathcal{H}=\{(p,q)\in \mathcal{A}: pq=1\}$, $\mathcal{A}^+=\{(p,q)\in \mathcal{A}: p>0\}$.

\section{The nonlinear eigenvalue problem}\label{sec:ne}

As mentioned in the introduction, the proof of Theorem~\ref{prop:convergence direct} relies on a good understanding of the properties of the nonlinear eigenvalue problem \eqref{lambdapq}. Note that \eqref{lambdapq} is defined for $p>0$. For $(p,q)\in \mathcal{A}$ with $p=0$ (i.e., satisfying \eqref{p=0}), we set 
\begin{equation}\label{nehari2} 
\mathcal{M}_q:=\left\{ u \in W_\nu^{2, \frac{q+1}{q}}(\Omega): \Big| |\{ u >0\}| - |\{ u < 0 \}| \Big | \le |\{ u=0 \} | \right \}. 
\end{equation}
 and 
\begin{align}
\Lambda_{0,q}:=\inf\left\{\frac{\|\Delta u\|_\frac{q+1}{q}}{\|u\|_{1}}:\ u  \in \mathcal{M}_q \setminus \{ 0\} \right\}
=\inf \left \{ \|\Delta u\|_{\frac{q+1}{q}} : \ u \in \mathcal{M}_q, \ \norm{u}_{1}=1 \right\}.\label{Lambda0q}
\end{align}
The goal of this section and the following is to show the following two results.  First, the next lemma establishes the relationship between $c_{p, q}$ and $\Lambda_{p, q}$.
\begin{lemma}\label{lemma lambda c}
If $(p, q) \in \mathcal{A} \setminus \mathcal{H}$, then 
\begin{align*}
\Lambda_{p, q}^{\frac{(p+1)(q+1)}{pq-1}} =\frac{(p+1)(q+1)}{pq-1} c_{p, q}.     
\end{align*}
Moreover, $u$ is a optimizer for $\Lambda_{p, q}$ with $\norm{u}_{p+1}=1$ if and only if 
\begin{equation} (\Lambda_{p, q}^{\frac{q+1}{pq-1}} u,- \, \Lambda_{p, q}^{\frac{p+1}{pq-1}} v),\quad \text{ with } \quad v:=\Lambda_{p,q}^{-\frac 1 q} |\Delta u|^{\frac 1 q- 1} \Delta u, \end{equation}
is an optimizer for $c_{p, q}$. 
\end{lemma}

The following is our main existence and convergence result for $\Lambda_{p, q}$, which implies Theorem~\ref{prop:convergence direct} and is a result of independent interest.

\begin{theo}\label{main thm lambda}
   Let $N\geq 1$, $\Omega\subset \R^N$ be a bounded smooth domain, and $(p,q)\in \mathcal{A}$. Then $\Lambda_{p,q}$ is achieved at some $u=u_{p,q}$ which is a weak solution of
    \begin{equation}\label{eq lambda}
    \Delta (|\Delta u|^{\frac 1 q -1} \Delta u) = \Lambda_{p, q}^{\frac{q+1}{q}} |u|^{p-1} u, \quad \partial_\nu u=\partial_\nu(|\Delta u|^{\frac 1 q -1} \Delta u)=0 \text{ on } \partial \Omega,\qquad \text{if $p>0$,}
    \end{equation}
    or
    \begin{equation}\label{eq lambda_0}
    \Delta (|\Delta u|^{\frac 1 q -1} \Delta u) = \Lambda_{p, q}^{\frac{q+1}{q}} \sign(u), \quad \partial_\nu u=\partial_\nu(|\Delta u|^{\frac 1 q -1} \Delta u)=0 \text{ on } \partial \Omega,\qquad \text{if $p=0$},
    \end{equation}    
and the map 
\[
\mathcal{A}\to \R^+,\qquad (p,q)\mapsto \Lambda_{p,q}
\]
is continuous.       Let $(p_n,q_n), (p,q)\in \mathcal{A}$ with $(p_n,q_n)\to (p,q)$ as $n\to\infty$ and let $u_{p_n,q_n}$ be an optimizer for $\Lambda_{p_n,q_n}$ with $\|u_{p_n,q_n}\|_{p_n+1}=1$. Then, there exists $u_{p,q}$, an optimizer for $\Lambda_{p,q}$, such that
\[
u_{p_n,q_n}\to u_{p,q} \quad \quad \text{ strongly in } C^{2,\zeta}(\overline \Omega) 
\]
for any $\zeta \in (0, \min\{ q , 1\})$. Moreover, if $p>0$, then 
$\Delta u_{p_n, q_n} \to \Delta u_{p, q}$ strongly in $C^{2, \zeta}(\overline \Omega),$ where
$\zeta \in (0, \min\{ p , 1\})$.
If $p=0$, then $\Delta u_{p_n, q_n} \to \Delta u_{p, q}$ strongly in $C^{1, \zeta}(\overline \Omega)$  for any $\zeta \in (0, 1)$. 
\end{theo}

A function $u \in W^{2, \frac{q+1}{q}}_\nu(\Omega)$ is a weak solution of \eqref{eq lambda} if
\[
\int_\Omega  |\Delta u|^{\frac 1 q -1} \Delta u \Delta \varphi=\Lambda_{p,q}^\frac{q+1}{q}\int_\Omega |u|^{p-1}u\varphi \qquad \text{ for every } \varphi\in W^{2,\frac{q+1}{q}}_\nu(\Omega)
\]
and an analogous definition is given for \eqref{eq lambda_0}.  Let
\[
f_p(t)=\begin{cases}
|t|^{p-1} t, & \text{ if } p>0,\\
\sign(t),& \text{ if } p=0.
\end{cases}
\]

We begin by showing the first statement of Theorem~\ref{main thm lambda}.

\begin{prop}\label{prop:existence_Lambda}
    The level $\Lambda_{p,q}$ is achieved whenever $(p,q)\in \mathcal{A}$ and, if 
 $u=u_{p,q}$ is an optimizer, then 
    \begin{equation}\label{eq lambda2}
    \Delta (|\Delta u|^{\frac 1 q -1} \Delta u) = \Lambda_{p, q}^{\frac{q+1}{q}} f_p(u), \quad \partial_\nu u=\partial_\nu(|\Delta u|^{\frac 1 q -1} \Delta u)=0 \quad \text{ on } \partial \Omega.
    \end{equation}
Moreover, given any  $u \in W^{2, \frac{q+1}{q}}_\nu(\Omega)$ solution of \eqref{eq lambda2}, then $(u, v):=(u,-|\Delta u|^{\frac 1 q-1} \Delta u)$ satisfies
    \begin{equation}\label{system scaling}
        -\Delta u = |v|^{q-1} v  \text{ in } \Omega,\quad
        -\Delta v = \Lambda_{p, q}^{\frac{q+1}{q}} f_p(u)  \text{ in } \Omega,\quad 
        \partial_\nu u =\partial_\nu v=0  \text{ on } \partial \Omega,
    \end{equation}
    In particular, $u \in C^{2, \zeta}(\overline \Omega)$, with $\zeta \in (0, \min\{q, 1\})$; whereas $v \in C^{2, \eta}(\overline \Omega)$, for $\eta\in (0,\min\{p, 1\})$ if $p \ne 0$, and $v \in C^{1, \zeta}(\overline \Omega)$, for any $\zeta\in (0,1)$ if $p = 0$.
\end{prop}

\begin{remark}[scalings]\label{eq:scalings_remark}
We point out that, if $(u,v)$ solves \eqref{system scaling}, then $(\bar U,\bar V):=(u,\Lambda_{p,q}^{-\frac{1}{q}}v)$ solves
\[
-\Delta \bar U=\Lambda_{p,q} |\bar V|^{q-1}\bar V,\quad -\Delta \bar V=\Lambda_{p,q} |\bar U|^{p-1} \bar U \text{ in } \Omega
\]
and, for $pq\neq 1$ and $( U,V):=(\Lambda_{p,q}^\frac{q+1}{pq-1}u,\Lambda_{p,q}^\frac{q+1}{q(pq-1)}v)$,
\[
-\Delta U= |V|^{q-1}V,\quad -\Delta V= |U|^{p-1}U \text{ in } \Omega
\]
\end{remark}

In the following two subsections, we separately consider the case $p>0$  and the case $p=0$, and we collect all the results needed in the proof of the existence part in Proposition 
\ref{prop:existence_Lambda}. Recall from the notation the definition of $\mathcal{A}^+$.
In Subsection~\ref{proof prop}, we take into account the regularity statement, and we then complete the proof of Proposition 
\ref{prop:existence_Lambda}. 

Both in the case $(p, q) \in \mathcal{A}_+$ and $(p, q) \in \mathcal{A} \setminus \mathcal{A}_+$ we will repeatedly use the following result 
(for the proof of existence, uniqueness and $W^{2, s}$ regularity see  \cite[Theorem and Lemma in page 143]{Rassias} or \cite[Theorem 15.2]{Agmon}, whereas we refer to \cite{Nardi} for the Schauder-type estimates). 

\begin{lemma}\label{lemma:regularity}
 If $s>1$, $\Omega$ be a smooth bounded domain in $\R^N$, and $h\in L^s(\Omega)$ with $\int_\Omega h=0$. Then there is a unique strong solution $u\in W^{2,s}(\Omega)$ of 
 \begin{align}\label{Nprob}
  -\Delta u = h\quad \text{ in }\Omega,\qquad \partial_\nu u=0\quad \text{ on }\partial \Omega,\qquad \int_\Omega u = 0.
  \end{align}
Moreover, there exists $C(\Omega,s)=C>0$ such that $\|u\|_{W^{2,s}}\leq C\|h\|_s.$

Furthermore, assume $h \in C^{0, \zeta}(\overline \Omega)$ for some $\zeta \in (0,1)$. Then the solution $u$ of \eqref{Nprob} is in $C^{2, \zeta}(\overline \Omega)$ and there exists $C'(\Omega, \zeta)=C'>0$ such that 
\[ \norm{u }_{C^{2, \zeta}} \le C' \norm{h}_{C^{0, \zeta}}. \]
\end{lemma}

\begin{coro}\label{Coro:equivalentnorms}
    For $s>1$, $u\mapsto \|u\|_1+\|\Delta u\|_s$ is a norm in $W^{2,s}(\Omega)$, equivalent to the standard one.
\end{coro}
\begin{proof}
   Let $u\in W^{2,s}(\Omega)$. The estimate $\|u\|_1+\|\Delta u\|_s\leq C\|u\|_{W^{2,s}}$ is immediate. As for the other one, the function $v:=u-\frac{1}{|\Omega|}\int_\Omega u$ satisfies \eqref{Nprob} for $h=\Delta u$, hence $\|u\|_{W^{2,s}}-|\Omega|^{\frac{1}{s}-1}\int_\Omega |u| \leq  \|v\|_{W^{2,s}}\leq C\|\Delta u\|_s$.
\end{proof}

\subsection{The case $(p,q)\in \mathcal{A}_+$}\label{sec: lambda p>0}
The proof of Proposition~\ref{prop:existence_Lambda} for $(p,q)\in \mathcal{A}_+$ is based, in most part, on the papers \cite{ST,PST}. In these papers a dual method was shown to be suitable to deal with the problem. To put ourselves in the same setting, we introduce some preliminary notions, introducing a dual level $D_{p,q}$ in \eqref{equivalent_def_D}, and showing its relation with $\Lambda_{p,q}$ in Lemma~\ref{lem:lambda d} below. For $s>1$, we define the operator $K : X^s \to W^{2, s}(\Omega)$ such that $Kh:=u$ if and only if 
\[
-\Delta u =h  \text{ in } \Omega, \qquad
\partial_\nu u=0  \text{ on } \partial \Omega,\qquad
\int_\Omega u=0, \]
where
\[ 
X^s=\left \{ f \in L^s(\Omega): \, \int_\Omega f=0 \right \}. 
\]
The operator $K$ is well defined and continuous, as a consequence of Lemma~\ref{lemma:regularity}. 
Also, for $t>0$ we define $K_t : X^{\frac{t+1}{t}}(\Omega)  \rightarrow W^{2, \frac{t+1}{t}}(\Omega)$ given by 
\begin{equation} \label{eq:def_kappap}
K_t h=K h+ \kappa_t(h) \, \text{ for some } \, \kappa_t(h) \in \R \, \text{ such that } \, \int_\Omega | K_t h|^{t-1} K_t h =0,
\end{equation}
i.e., for $h\in X^\frac{t+1}{t}$,  $u=K_t h$ is the unique (strong) solution of
\[
-\Delta u =h  \text{ in } \Omega, \qquad
\partial_\nu u=0  \text{ on } \partial \Omega,\qquad
\int_\Omega |u|^{t-1}u=0.
\]
For extra properties of $\kappa_t(h)$ we refer to \cite{PW}. We set 
\[
X:=X^\alpha \times X^\beta,\qquad \text{ for }\alpha=\frac{p+1}p, \quad \beta=\frac{q+1}q,
\]
and define $\gamma_1,\gamma_2$ so that
$\gamma_1+\gamma_2=1$, $\gamma_1\alpha+\gamma_2\beta=:\gamma$,
which gives
\[
\gamma_1=\frac{\beta}{\alpha+\beta}=\frac{p(q+1)}{2pq+p+q},\quad \gamma_2= \frac{\alpha}{\alpha+\beta}=\frac{q(p+1)}{2pq+p+q},\quad\gamma=\frac{(p+1)(q+1)}{2pq+p+q},\quad \text{ and } \frac{1}{\alpha}+\frac{1}{\beta}=\frac{1}{\gamma}.
\]
We introduce the level:
\begin{align}
D_{p, q} :=&\sup \left \{ \int_\Omega f Kg: \quad  (f, g) \in X, \quad \gamma_1 \norm{f}_\alpha^\alpha + \gamma_2 \norm{g}_\beta^\beta =1 \right \}\\
\label{equivalent_def_D} =& \sup_{(f,g)\in X\setminus \{(0,0)\}} \frac{\displaystyle\int_\Omega fKg}{\left(\gamma_1 \norm{f}_\alpha^\alpha + \gamma_2 \norm{g}_\beta^\beta\right)^{\frac{1}{\gamma}}} =\mathop{\sup_{(f,g)\in X}}_{f,g\neq 0}\frac{\displaystyle \int_\Omega f Kg}{\|f\|_\alpha \|g\|_\beta}.
\end{align}
To see that $D_{p,q}<\infty$ and to check its equivalent characterizations, see \cite[pp. 755--756]{PST}; observe that a key point is the Young's inequality:
\begin{equation*}
(\|f\|_\alpha \|g\|_\beta)^\gamma \leq \gamma_1 \|f\|_\alpha^\alpha+\gamma_2 \|g\|_\beta^\beta\qquad \text{ for every } (f,g)\in X.
\end{equation*}

We now recall some results whose proofs can be found in the literature. Let $S_{p,q}$ be the best Sobolev constant for the embedding $\mathcal{D}^{2, {\frac{q+1}{q}}}(\R^N) \hookrightarrow L^{p+1}(\R^N)$, namely
\begin{equation}\label{eq:Spq}
S_{p,q}= \inf\left\{\|\Delta u\|_{L^{\frac{q+1}{q}}(\R^N)}:\ u\in \mathcal{D}^{2,\frac{q+1}{q}}(\R^N),\ \|u\|_{L^{p+1}(\R^N)}=1 \right\},
\end{equation} 
where
we denote by $\mathcal{D}^{2, \eta}(\R^N)$ the completion of $C^\infty_c(\R^N)$ with respect to the norm $\|\Delta u\|_{L^\eta(\R^N)}$. 

\begin{lemma}\label{lemma maximizing sequences}
    Assume $(p, q)$ satisfies 
    \begin{equation}\label{subcritical lemma}
        \frac{1}{p+1} + \frac{1}{q+1}>\frac{N-2}{N},
    \end{equation}
    or
    \begin{equation}\label{critical lemma}
        \frac{1}{p+1} + \frac{1}{q+1}=\frac{N-2}{N} \quad \text{ and }\quad  D_{p,q} > \frac{2^{2/N}}{S_{p,q} }.
     \end{equation}
Let $(f_n, g_n) \in X $ be a maximizing sequence, namely $\int_\Omega f_n K g_n \to D_{p, q}$, with $\gamma_1 \norm{f_n}_\alpha + \gamma_2 \norm{g_n}_\beta =1$.
Then,  there exists $(f, g) \in X$ such that 
\begin{itemize}
    \item[(i)] $f_n \to f$ in $L^\alpha(\Omega)$ and $g_n \to g$ in $L^\beta(\Omega)$;
    \item[(ii)] $\int_\Omega fKg=D_{p,q}$, and $\gamma_1\|f\|_\alpha^\alpha+\gamma_2\|g\|_\beta^\beta=1$. In particular, $D_{p, q}$ is attained. 
\end{itemize}
\end{lemma}
\begin{proof}
In the subcritical case \eqref{subcritical lemma}, this is a simple consequence of the continuity of $K$ and the compactness of the embeddings
$W^{2,\alpha}(\Omega)\hookrightarrow L^{q+1}(\Omega)$, $W^{2,\beta}(\Omega)\hookrightarrow L^{p+1}(\Omega)$. As for the critical case \eqref{critical lemma}, the statement  follows from \cite[Lemma 3.1]{PST}.
\end{proof}

\begin{lemma}\label{lemma:Dpq_achieved}
Let $(f, g) \in X$ be any maximizer for $D_{p, q}$. Then, 
\begin{equation}\label{eq:dual_system}
K_pf = D_{p,q} |g|^{\frac 1 q -1} g, \quad K_q g=D_{p,q} |f|^{\frac 1 p -1} f \qquad \text{a.e. in } \Omega. \end{equation}
\end{lemma}
\begin{proof}
This can be found within the proof of \cite[Proposition 2.5]{PST}, namely at page 755. Even if the proof therein is given just for the critical case \eqref{critical lemma}, the same computations yield the conclusion also in the subcritical case \eqref{subcritical lemma}, and actually for any $(p, q)$ such that $D_{p,q} $ is attained.
\end{proof}

\begin{lemma}\label{lem:lambda d}
If $D_{p,q}$ is attained, then 
\[ \Lambda_{p, q}= D_{p, q}^{-1}. \]
Also, $u$ is an optimizer for $\Lambda_{p, q}$ such that $\norm{u}_{p+1}=1$ if, and only if, $(f, g)=(|u|^{p-1}u,\Lambda_{p, q}^{-1}(-\Delta u))$ is an optimizer for $D_{p, q}$.
\end{lemma}
\begin{proof}
Let $(f, g)\in X\setminus \{(0,0)\}$ attain $D_{p, q}$ and be such that $\gamma_1 \norm{f}_\alpha^\alpha +\gamma_2 \norm{g}_\beta^\beta=1$. 
Then, by Lemma~\ref{lemma:Dpq_achieved}, they satisfy 
\[ K_pf = D_{p,q} |g|^{\frac 1 q -1} g, \quad K_q g=D_{p,q} |f|^{\frac 1 p -1} f. \]
We define $u:=|f|^{\frac 1 p -1} f= D_{p,q}^{-1}(K_qg) \in W^{2,\beta}(\Omega)$, so that
\[
\int_\Omega |u|^{p-1}u=\int_\Omega f=0.
\]
Notice that $\norm{f}_\alpha^\alpha= \norm{g}_\beta^{\beta}=D_{p,q}^{-1}\int_\Omega fKg=1$, thus $\norm{u}_{p+1}=1$. Also, $\partial_\nu u= D_{p, q} \partial_\nu (Kg)=0$, hence $u\in W_\nu^{2,\beta}(\Omega)$.

We claim that $u$ is an optimizer for $\Lambda_{p, q}$. Indeed, take a minimizing sequence for $\Lambda_{p, q}$, namely $w_n \in W^{2, \beta}_\nu(\Omega)\setminus \{0\}$ such that $\norm{\Delta w_n}_{\beta} \to \Lambda_{p, q}$, $\norm{w_n}_{p+1}=1$ and $\int_\Omega |w_n|^{p-1}w_n=0$. We set $\tilde f_n:=|w_n|^{p-1}w_n$ and $\tilde g_n:=\Delta w_n /\norm{\Delta w_n}_{\beta}$. Since $-\Delta u = D_{p, q}^{-1} g$, 
\[ \Lambda_{p, q}^{\beta} \le  \int_\Omega |\Delta u|^{\beta} = D_{p, q}^{-\beta} \le  \lim_{n \to \infty}\left(\int_\Omega \tilde f_n K \tilde g_n \right)^{-\beta} = \lim_{n \to \infty} \norm{\Delta w_n}_{\beta}^{\beta}= \Lambda_{p, q}^{\beta}, \]
from which we get $\Lambda_{p, q}= D_{p, q}^{-1}$, and $u$ is an optimizer for $\Lambda_{p, q}$.

Let now $u \in W_\nu^{2, \beta}(\Omega)$ be such that $\|u\|_{p+1}=1$, $\int_\Omega |u|^{p-1}u=0$ and   $\|\Delta u\|_\beta=\Lambda_{p, q}$. Define $f:=|u|^{p-1}u \in L^{\alpha}(\Omega)$ and $g:=\Lambda_{p, q}^{-1} (-\Delta u) \in L^{\beta}(\Omega)$. Notice that $\int_\Omega f=\int_\Omega g=0$. Also, $\norm{f}_\alpha=1$, and
\[ K_pf = \Lambda_{p, q}^{-\beta} |\Delta u|^{\frac 1 q -1} \Delta u = \Lambda_{p, q}^{-1} |g|^{\frac 1 q -1} g, \quad K_qg=\Lambda_{p,q}^{-1}=\Lambda_{p, q}^{-1} |f|^{\frac 1 p -1} f. \]
Thus, in particular,
\[ 1=\norm{f}_\alpha^\alpha= \Lambda_{p, q} \int_\Omega f Kg=\Lambda_{p,q}\int_\Omega Kfg = \norm{g}_\beta^{\beta}, \quad \text{ hence } \quad 
\int_\Omega fKg=\Lambda_{p,q}^{-1}=D_{p,q}
\]
and $(f,g)$ is an optimizer for $D_{p,q}$.
\end{proof}

\subsection{The case $(p,q)\in \mathcal{A} \setminus  \mathcal{A}_+$}

Here we prove  Proposition~\ref{prop:existence_Lambda} for  $(p, q) \in \mathcal{A} \setminus  \mathcal{A}_+$, namely $(p, q)$ satisfying \eqref{p=0}: $p=0$ and $0<q<\infty$. We recall that in this case we have to minimize on the auxiliary set $\mathcal{M}_q$ defined in \eqref{nehari2}.

Let us begin with the following result, whose proof can be found in  \cite[Lemma 5.1]{PW} (replacing only $W^{1,2}(\Omega)$ by $W^{2,\beta}_\nu(\Omega)$ therein).

\begin{lemma}\label{lem:properties M}
\begin{itemize} 
\item[(i)] If $u \in \mathcal{M}_q$, then
\[ \int_\Omega |u+c| > \int_\Omega |u| \quad \text{ for every } c \ne 0. \]
\item[(ii)] For any $u \in W_\nu^{2, \beta}(\Omega)$, there exists a unique $c(u) \in \R$ such that $u + c(u) \in \mathcal{M}_q$. Moreover, the map $c:L^1(\Omega) \to \R$ defined as $u \mapsto c(u)$ is continuous. 
\item[(iii)] If $u,v \in  W_\nu^{2, \beta}(\Omega)$ satisfy $u \le v$, then $c(u) \ge c(v)$. 
\end{itemize}
\end{lemma}

Recall the definition of  $\Lambda_{0, q}$ given in \eqref{Lambda0q}. We first show the following
\begin{lemma}\label{lem: lambda 0 att}
    Let $(p,q)$ satisfy \eqref{p=0}. Then, $\Lambda_{0, q}$ is attained. 
\end{lemma}
\begin{proof}
We adapt the proof of \cite[Lemma 5.3]{PW} to our context. Let $u_n \in \mathcal{M}_q$ be a minimizing sequence for $\Lambda_{0, q}$, normalized in such a way that $\norm{u_n}_1=1$ (and so, in particular, $\norm{\Delta u_n}_{\beta}$ is uniformly bounded). Then, by Corollary~\ref{Coro:equivalentnorms}, $u_n$ is bounded in $W^{2, \beta}_\nu(\Omega)$. Up to a subsequence, $u_n \rightharpoonup u$ in $W^{2, \beta}(\Omega)$ and $u_n \to u$ in $L^1(\Omega)$. We check that $u\in \mathcal{M}_q$.

By compactness of the trace operator 
\[ \gamma: W^{2, \beta}(\Omega) \to L^t(\partial \Omega), \quad \frac{1}t > \frac{Nq - 2(q+1)}{(q+1)(N-1)}, \]
see \cite[Chapter 2.6.2]{Necas}, we conclude that $\partial_\nu u=0$ on $\partial \Omega$ in the sense of traces, and so $u \in W_\nu^{2, \beta}(\Omega)$. 

Next, we apply Lemma~\ref{lem:properties M} (i) to get
\[ \int_\Omega |u + c| = \lim_{n \to \infty} \int_\Omega |u_n +c| \ge \lim_{n \to \infty} \int_\Omega |u_n| = \int_\Omega | u| \quad \text{ for any } c\in \mathbb R. \]
Choose $\tilde c$ such that $ u + \tilde c \in \mathcal{M}_q$ (this is possible due to Lemma~\ref{lem:properties M} (ii)). Thus, assuming $\tilde c \ne 0$, we have
\[ \int_\Omega |u| = \int_\Omega |u + \tilde c- \tilde c| > \int_\Omega | u + \tilde c| \ge \int_\Omega |u|, \]
once again using Lemma~\ref{lem:properties M} (i). The contradiction yields $\tilde c =0$, namely $ u \in \mathcal{M}_q$.

Finally, observe that $\|\Delta u\|_\beta\leq \liminf \|\Delta u_n\|_\beta$, from which we get that $u$ is a minimizer for $\Lambda_{0, q}$. 
\end{proof}

Before showing that minimizers of $\Lambda_{0,q}$ are weak solutions of \eqref{eq lambda_0}, we need a density result. Let
\[
C^2_\nu(\overline \Omega):=\{v\in C^2(\overline \Omega):\ \partial_\nu v=0 \text{ on } \partial \Omega\}.
\]

\begin{lemma}\label{BrezisPonce_densitylemma}
Let $\psi\in C^2(\overline \Omega)$ with $\psi\geq 0$. There exists a sequence $(\zeta_k)\subset   C^2_\nu(\overline \Omega)$ such that $\zeta_k \ge 0$ and
\begin{equation}\label{eq:BP_densitylemma_aux}
\norm{\nabla \zeta_k}_\infty \le C, \quad \zeta_k \to \psi \text{ uniformly in } \Omega, \quad \nabla \zeta_k \to \nabla \psi \text{ a.e. in } \Omega. 
\end{equation}
\end{lemma}
\begin{proof}
This is a consequence of \cite[eqs. (2.8)--(2.9) and Proposition 2.2]{BrezisPonce}. For completeness, we provide more details here. Consider first the case $\psi>0$ in $\overline \Omega$.
Define $\Phi \in C^\infty_0(\R)$ such that $\Phi(t)=t$ for all $t \in [-1, 1]$.
Take 
\[
\zeta_k = \psi -\frac 1k \Phi(k \eta) \quad \text{ in } \overline \Omega,
\]
where  $\eta \in C^2(\bar{\Omega})$  with  $\eta=0$, $\frac{\partial \eta}{\partial n}=\frac{\partial \psi}{\partial n}$ on  $\partial \Omega$. Then $\zeta_k \in C^2_\nu(\overline \Omega)$ and \eqref{eq:BP_densitylemma_aux} is satisfied, see \cite{BrezisPonce}. 
Choose $k$ large enough such that $ \frac 1 k \norm{\Phi}_\infty < \min_{x \in \overline \Omega} \psi(x)$.
This implies $\zeta_k \ge0$ for any $k$ large enough. 

Finally, for any $\psi \in C^2(\overline \Omega)$ such that $\psi \ge 0$, we take $\psi_n = \psi + \epsilon_n$, where $\epsilon_n >0$, apply the above argument and let $\epsilon_n \to 0^+$ as $n \to +\infty$.
\end{proof}

In what follows we denote
\[ \sign_+(t):= \mathds{1}_{\{t \ge 0\}} - \mathds{1}_{\{t <0\}} \quad \text{ and } \quad \sign_-(t):= \mathds{1}_{\{t > 0\}} - \mathds{1}_{\{t \le 0\}}. \]
Observe that
\begin{equation}\label{eq:ineq_sign}
|a+b|\geq |a|+\sign_-(a)b \qquad \text{ for every } a,b\in \R,
\end{equation}
and $u \in \mathcal{M}_q$ implies
\[ \int_\Omega \sign_-(u) \le 0 \le \int_\Omega \sign_+(u). \]

\begin{lemma}\label{lemma:optimizersLambda_0q_are_wsol}
Let $u \in \mathcal{M}_q$ be an optimizer for $\Lambda_{0, q}$. Then, it is a weak solution to~\eqref{eq lambda_0}. 
\end{lemma}
\begin{proof}
We draw inspiration to this proof from the one of \cite[Lemma 5.3]{PW}. In our context of higher order operator is, however, more delicate, since the space $W^{2, \beta}(\Omega)$ is not closed for the operation of taking the positive and negative parts of functions. Instead, we have to rely on the density results from Lemma~\ref{BrezisPonce_densitylemma} and on a regularity result from \cite{BrezisPonce}. From now on, without loss of generality, we take $u \in \mathcal{M}_q$ an optimizer for $\Lambda_{0, q}$ with $\|u\|_1=1$. 

\smallbreak

\noindent Step 1. We claim that 
\begin{equation}\label{claim} \int_\Omega \sign_-(u) \psi \le \Lambda_{0, q}^{- \beta}  \int_\Omega |\Delta u|^{\frac 1 q -1} \Delta u \Delta \psi \le \int_\Omega \sign_+(u) \psi \quad \text{ for all } \psi \in W_\nu^{2, \beta}(\Omega), \psi \ge 0. \end{equation}
Indeed, assume by contradiction that there exists $\psi \in W_\nu^{2, \beta}(\Omega)$ with $\psi \ge 0$ such that 
\begin{equation}\label{contrad}
\int_\Omega \sign_-(u) \psi >  \Lambda_{0, q}^{-\beta}\int_\Omega |\Delta u|^{\frac 1 q -1} \Delta u \Delta \psi. 
\end{equation}

Then, for every $c \le 0$ and $t \ge 0$, using \eqref{eq:ineq_sign} with $a=u$ and $b=c+t\psi$, recalling that $\int_\Omega \sign_-(u) \le 0$,  we have that
\[ \norm{u+c+t\psi}_1 \ge \norm{u}_1  + t \int_\Omega \sign_-(u) \psi=1+ t \int_\Omega \sign_-(u) \psi. \]
Now, $c(u+t\psi) \le c(u)=0$ for any $t \ge 0$, by Lemma~\ref{lem:properties M} (iii). 
Moreover, by a Taylor expansion, 
\begin{align*}
\frac{\left(\int_\Omega |\Delta (u+t\psi)|^{\beta}\right)^{\frac{1}{\beta}}}{\int_\Omega |u+t\psi+c(u+t\psi)|} & \le \frac{\left(\int_\Omega |\Delta (u+t\psi)|^{\beta}\right)^\frac{1}{\beta}}{1 + t \int_\Omega \sign_-(u)\psi } \\
=&\Lambda_{0,q}+t\left( \|\Delta u\|_\beta^{-1/q} \int_\Omega |\Delta u|^{\frac{1}{q} - 1} \Delta u \Delta \psi  -  \|\Delta u\|_\beta  \int_\Omega \sign_-(u)\psi\right)+o(t)\\
=&\Lambda_{0,q}+t\|\Delta u\|_\beta 
 \left( \Lambda_{0,q}^{-\beta}   \int_\Omega |\Delta u|^{\frac{1}{q} - 1} \Delta u \Delta \psi  -    \int_\Omega \sign_-(u)\psi \right) +o(t)
\end{align*}
The coefficient of the first order term is negative by \eqref{contrad}, which gives
\[ 
\frac{\left(\int_\Omega |\Delta (u+t\psi)|^{\beta}\right)^{\frac{1}{\beta}}}{\int_\Omega |u+t\psi+c(u+t\psi)|} < \Lambda_{0, q} 
\]
for $t$ small enough, a contradiction. The other inequality in \eqref{claim} follows analogously.

\smallbreak 

Step 2. We consider ~\eqref{claim} for test functions  $\psi \in C^2_\nu(\overline \Omega), \psi \ge 0$; in particular,
for 
\begin{align*}
v:= -\Lambda_{0, q}^{- \beta} |\Delta u|^{\frac 1 q-1}\Delta u \in L^{q+1}(\Omega)\subset L^1(\Omega)
\qquad \text{we have}\qquad 
\sup_{\substack{\psi \in C^2_\nu(\overline \Omega)\\ \norm{\psi}_\infty \le 1}} \abs{\int_\Omega v \Delta \psi} < \infty.
\end{align*}
Thus, we can apply Proposition 2.1 in \cite{BrezisPonce} and get $v \in W^{1, t}(\Omega)$ for every $1 \le t < \frac{N}{N-1}$. 
Therefore,
\[ \int_\Omega \sign_-(u) \psi \le  \int_\Omega \nabla v  \nabla \psi \le \int_\Omega \sign_+(u) \psi \quad \text{ for all } \psi \in C^2_\nu(\overline \Omega), \psi \ge 0. \]

By the density result in Lemma~\ref{BrezisPonce_densitylemma}, these inequalities are true for every $\psi\in C^2(\overline \Omega)$, $\psi\geq 0$. Since the extension operator of $W^{1,\beta}(\Omega)$ to $W^{1,\beta}(\R^N)$ and the operation of convolution keep the sign, therefore
\begin{equation}\label{bounds}
\int_\Omega \sign_-(u) \psi \le  \int_\Omega \nabla v  \nabla \psi \le \int_\Omega \sign_+(u) \psi \quad \text{ for all } \psi \in W^{1, \beta}(\Omega), \psi \ge 0,  \end{equation}
again by a density argument. This implies
\[ \abs{\int_\Omega \nabla v  \nabla \psi}  \le \int_\Omega \sign_+(u) \psi - \int_\Omega \sign_-(u) \psi, 
 \quad \text{ for all } \psi \in W^{1, \beta}(\Omega), \psi \ge 0. \]
We now exploit these inequalities to get
\begin{align} \abs{\int_\Omega \nabla v \nabla \psi }\le \abs{ \int_\Omega \nabla v  \nabla \psi^+} +\abs{\int_\Omega \nabla v  \nabla \psi^-} &\le  \int_\Omega \sign_+(u) |\psi| - \int_\Omega \sign_-(u) |\psi|  \noindent \\
& \le 2 \int_\Omega |\psi|, \quad \text{ for all } \psi \in W^{1, \beta}(\Omega). \label{eq:L^1continuity}
\end{align}

Let $f(v):C^\infty_c(\Omega) \to \R$ be defined as
\[ f(v)\psi=\int_\Omega \nabla v \nabla \psi, \]
which is linear and continuous for the $L^1(\Omega)$-norm, by \eqref{eq:L^1continuity}. Then $f(v)$ can be uniquely extended to an element in $(L^1(\Omega))'$, and there exists $\eta\in L^\infty(\Omega)$ such that
\[ 
f(v)\psi= \int_\Omega \eta \psi \quad \text{ for all }\psi \in L^1(\Omega). 
\]
Combining this with~\eqref{bounds} gives
\[
\int_\Omega \sign_-(u) \psi \le  \int_\Omega \eta\psi \le \int_\Omega \sign_+(u) \psi\qquad \text{for all $\psi\in C^\infty_c(\Omega)$, $\psi\ge 0$,}
\]
hence
\[ \sign_-(u) \le \eta \le \text{sign}_+(u) \quad \text{a.e. in } \Omega. \] 
From this, $\eta=\sign(u)$ whenever $u\neq 0$.  We now check that actually $\eta=\sign(u)$ a.e. in $\Omega$.

The function $v$ is a distributional solution of $-\Delta v=\eta$, and $\eta\in L^\infty(\Omega)\subset L^2(\Omega)$; then, actually $v\in H^2_{loc}(\Omega)$ (see \cite[Lemma (6.33)]{FollandPDE}) and $v$ is a weak solution, so (now by classical regularity theory for weak solutions)  $v \in W^{2, t}_{loc}(\Omega)$ for any $t\geq 1$, hence $v\in C_{loc}^{1,\gamma}(\Omega)$. Again by regularity theory, since $-\Delta u=\Lambda_{0, q}^{\beta}|v|^{q-1}v$, one has $u \in C_{loc}^{2,\gamma'}(\Omega)\subset W^{2,1}_{loc}(\Omega)$. Therefore, $\nabla u=0$  a.e. on $\{ u=0\}$ (by \cite[Lemma 7.6]{GT}) and also $\Delta u=\operatorname{div}(\nabla u)=0$ a.e. in $\{u=0\}$. 
In conclusion, $v=0$ a.e. in $\{u=0\}$ and, repeating the previous reasoning, also $\Delta v=0$ in this set.   In conclusion, $\eta=\sign(u)$ a.e. in $\Omega$, and 
\[ \int_\Omega \nabla v \nabla \psi = \int_\Omega \sign(u) \psi, \quad \text{ for all } \psi \in C^\infty(\overline \Omega). \]
By density and integration by parts:
\[
\Lambda^{-\beta}_{0,q}\int_\Omega  |\Delta u|^{\frac 1 q -1} \Delta u \Delta \psi- \Lambda_{0, q}^{- \beta}\int_{\partial \Omega } |\Delta u|^{\frac 1 q-1}\Delta u \psi=\int_\Omega \sign(u)\psi \qquad \text{ for every } \psi\in W^{2,\beta}(\Omega).
\]
Taking now $\psi\in W^{2,\beta}_\nu(\Omega)$, we conclude that $u$ is a weak solution to~\eqref{eq lambda}. 
\end{proof}

\subsection{Proof of Proposition~\ref{prop:existence_Lambda}}\label{proof prop}
In order to prove the regularity statement in Proposition~\ref{prop:existence_Lambda}, it is useful to read \eqref{eq lambda2} as a system, and apply Lemma~\ref{lemma:regularity}. 
Let us start by establishing the following 

 \begin{prop}\label{equivalence}
Let us assume $p \ge 0$ and $q >0$. 
Let $u \in W^{2, \beta}_\nu(\Omega)$ be a weak solution for \eqref{eq lambda2}, then $(u, v):=(u, -|\Delta u|^{\frac 1 q-1} \Delta u)$ satisfies \eqref{system scaling}. In particular, $u \in C^{2, \zeta}(\overline \Omega)$, with $\zeta \in (0, \min\{q, 1\})$; whereas $v \in C^{2, \eta}(\overline \Omega)$, for $\eta\in (0,\min\{p, 1\})$ if $p \ne 0$, and $v \in C^{1, \zeta}(\overline \Omega)$, for any $\zeta\in (0,1)$ if $p = 0$.
\end{prop}

\begin{proof}
{Case $p>0$}.    We adapt the proof in \cite{dosSantos}, which deals with Navier boundary conditions. We aim to check that $(u,v)$ are a strong solution. Let us preliminary notice that $\int_\Omega |u|^{p-1}u=0$, thus $|u|^{p-1}u\in X^\alpha$.
    By Lemma~\ref{lemma:regularity}, there exists a unique $w\in  W^{2, \alpha}_\nu(\Omega) \hookrightarrow L^{q+1}(\Omega)$  such that 
    \[ -\Delta w = \Lambda_{p, q}^{\frac{q+1}{q}} |u|^{p-1}u \text{ in } \Omega, \quad \partial_\nu w=0 \text{ on } \partial \Omega, \quad \int_\Omega |w|^{q-1}w=0, \]
    namely $w:=K_q(\Lambda_{p, q}^{\frac{q+1}{q}} |u|^{p-1}u)$.  Now,  we can deduce that there exists a unique strong solution $z\in W_\nu^{2,\beta}$ such that 
    \[ -\Delta z = |w|^{q-1} w \text{ in } \Omega, \quad \partial_\nu z=0 \text{ on } \partial \Omega, \quad \int_\Omega |z|^{p-1}z= 0.
    \]
Since, for any $\varphi \in W^{2, \beta}_\nu(\Omega)$,
\begin{align*}
    \int_\Omega |\Delta z|^{\frac 1 q -1} \Delta z \Delta \varphi&= - \int_\Omega w \Delta \varphi = \int_\Omega \nabla w \nabla \varphi= -\int_\Omega\Delta w \varphi=\Lambda_{p, q}^{\frac{q+1}{q}} \int_\Omega |u|^{p-1}u \varphi= \int_\Omega |\Delta u|^{\frac 1q-1} \Delta u \Delta \varphi, 
\end{align*}
By Lemma~\ref{lemma:regularity}, for any $\eta  \in L^{\beta}(\Omega)$ with $\int_\Omega \eta=0$, there exists a $\psi \in W^{2, \beta}_\nu(\Omega)$ such that  $\int_\Omega \psi=0 $ and $-\Delta \psi = \eta$ in $\Omega$, $\partial_\nu \psi=0$ on $\partial \Omega$. 
This implies that 
\begin{equation}\label{eq isometry} \int_\Omega \tilde z \eta = \int_\Omega \tilde u \eta \quad \text{ for any } \eta \in \left \{ \psi \in L^{\beta}(\Omega): \int_\Omega \psi=0\right\},  \end{equation}
where $\tilde z:= |\Delta z|^{\frac 1 q -1} \Delta z, \quad \tilde u:= |\Delta u|^{\frac 1q-1} \Delta u$.
One has
\[ \tilde z, \tilde u \in \left \{ \psi \in L^{q+1}(\Omega): \int_\Omega |\psi|^{q-1} \psi=0 \right\} \cong \left \{ \psi \in L^{\beta}(\Omega) : \int_\Omega \psi=0 \right\}',  \]
 since 
\[ \int_\Omega ||\Delta z |^{\frac1q-1} \Delta z|^{q-1} |\Delta z |^{\frac1q-1} \Delta z =- \int_\Omega |w|^{q-1} w= 0 = \int_\Omega \Delta u=\int_\Omega ||\Delta u|^{\frac 1 q-1} \Delta u|^{q-1} |\Delta u|^{\frac 1 q-1} \Delta u  \]
(notice that we used the fact that $\partial_\nu u=0$ to conclude $\int_\Omega \Delta u=0$). Recalling~\eqref{eq isometry} we conclude that $|\Delta u|^{\frac 1 q-1} \Delta u = |\Delta z |^{\frac1q-1} \Delta z$. 
Therefore, it is easy to see that 
\[ \Delta (u-z)=0 \quad \text{ in } \Omega,\qquad \partial_\nu u=\partial_\nu z=0 \text{ on } \partial \Omega,\quad \int_\Omega |u|^{p-1}u=\int_\Omega |z|^{p-1}z=0,\]
so that 
$z=u$, and also $v=w$.
Now that we know that $(u,v)$ is a strong solution of~\eqref{system scaling}, the regularity statement follows directly from \cite[Proposition 1.3]{PST} and \cite[Theorem 1.1]{ST2}.

\smallbreak

In the case $p=0$, an application of Lemma~\ref{lemma:regularity} yields the existence of a unique $w$ which belongs to  $W^{2, r}_\nu(\Omega)\hookrightarrow L^{q+1}(\Omega)$ for any $r \ge 1$ such that 
    \[ -\Delta w = \Lambda_{0, q}^{\frac{q+1}{q}} \sign(u) \text{ in } \Omega, \quad \partial_\nu w=0 \text{ on } \partial \Omega, \quad \int_\Omega |w|^{q-1}w=0. \]
Moreover, we choose $z\in W_\nu^{2,\beta}(\Omega)$ such that 
\[ -\Delta z = |w|^{q-1} w \text{ in } \Omega, \quad \partial_\nu z=0 \text{ on } \partial \Omega, \quad z \in \mathcal{M}_q.
    \] 
As above, we can conclude that  $\Delta u=\Delta z$. 
 By Lemma~\ref{lemma:regularity} we know that $u=z + c$ for some $c \in \R$, and by Lemma~\ref{lem:properties M} and since $u \in \mathcal{M}_q$, we get that $c=c(z)=0$, thus, $u=z$. 
 
To conclude, let us prove the regularity statement.
We now have that $v:=|\Delta u|^{\frac 1 q-1} \Delta u$ is a strong solution of 
\[ -\Delta v =\Lambda_{0, q}^{\frac{q+1}{q}} \sign(u) \text{ in } \Omega \quad \partial_\nu v=0 \text{ on } \partial \Omega, \]
thus by Lemma~\ref{lemma:regularity} we conclude $v \in W^{2, r}(\Omega)$ for any $r \ge 1$, and $v \in C^{1, \zeta}(\overline \Omega)$ for any $\zeta\in (0,1)$. Since
\[ -\Delta u = |v|^{q-1} v \text{ in } \Omega \quad \partial_\nu u=0 \text{ on } \partial \Omega, \]
again Lemma~\ref{lemma:regularity} gives that $u \in C^{2, \zeta}(\overline \Omega)$ for some $\zeta >0$. 
\end{proof}

Now, we are ready to prove Proposition~\ref{prop:existence_Lambda}.
\begin{proof}[Proof of Proposition~\ref{prop:existence_Lambda}]
Let us first consider the case $(p, q) \in \mathcal{A}_+$. Under the assumption \eqref{CH2} the condition $D_{p,q} > \frac{2^{2/N}}{S_{p,q} }$ is verified, see \cite[Theorem 1.2]{PST}.
    In particular, by Lemma~\ref{lemma maximizing sequences}, given $(p, q) \in \mathcal{A}_+$, the level $D_{p, q}$ is attained. Equation \eqref{eq lambda2} is now a direct combination of Lemmas~\ref{lemma:Dpq_achieved} and~\ref{lem:lambda d}. 
    
As for the case $p=0$, the conclusion follows combining Lemmas~\ref{lem: lambda 0 att} and~\ref{lemma:optimizersLambda_0q_are_wsol}, and once again Proposition~\ref{equivalence} to get the regularity.
\end{proof}

\section{Convergence results for $\Lambda_{p,q}$ and its associated optimizers. }\label{sec:conv1}
\subsection{Convergence for $(p, q) \in \mathcal{A}_+$}

In what follows, we prove that $D_{p, q}$ (defined in \eqref{equivalent_def_D}) is continuous with respect to $(p, q)$.

\begin{prop}\label{D cont}
Let $(p, q) \in \mathcal{A}_+$. 
Moreover, let $(p_n, q_n) \in \mathcal{A}_+$, be such that $p_n \to p$, $q_n \to q$.
Then 
\[ D_{p_n, q_n} \to D_{p, q}.\]
Moreover, if $(f_n, g_n) \in X^{\alpha_n}\times X^{\beta_n}$ is such that $\gamma_{1,n}\|f_n\|^{\alpha_n}_{\alpha_n}+\gamma_{2,n}\|g_n\|_{\beta_n}^{\beta_n}=1$, for
\[
\alpha_n=\frac{p_n+1}{p_n},\quad \beta_n=\frac{q_n+1}{q_n},\quad \gamma_{1,n}=\frac{\beta_n}{\alpha_n+\beta_n},\quad \gamma_{2,n}=\frac{\alpha_n}{\alpha_n+\beta_n}.
\]
and achieves $D_{p_n, q_n}$,  then there exists $(f, g) \in X$ achieving $D_{p, q}$ such that, up to a subsequence
\[ (f_n, g_n) \to (f, g) \quad \quad \text{ strongly in } L^\alpha(\Omega) \times L^\beta(\Omega). \]
\end{prop}

\begin{proof}
We split the proof into two parts. For simplicity of notation, we call 
\[ D_n:=D_{p_n, q_n},  \quad \text{ and } \quad X_n:=X^{\alpha_n}\times X^{\beta_n}. \]
Note that, under our assumptions, by Lemma~\ref{lemma maximizing sequences}, $D_{p, q}$ and $D_n$ are both achieved.
\smallbreak

\noindent \textit{Step 1.} We prove that $\liminf_{n \to \infty} D_n \ge D_{p, q}$.
Take $(f, g )\in X$ achieving $D_{p, q}$. 
Thus $u:=D_{p, q}^{-1} K_p g$, and $v:=D_{p, q}^{-1}K_q f$ satisfy
\[ u= |f|^{1/p-1}f, \quad v=|g|^{1/q-1} g, \]
and 
\[ -\Delta u = D_{p, q}^{-1} g= D_{p, q}^{-1} |v|^{q-1}v , \quad  -\Delta v = D_{p, q}^{-1} f= D_{p, q}^{-1} |u|^{p-1}u , \quad \partial_\nu u=\partial_\nu v=0 \text{ on } \partial \Omega, \]
recall Lemma~\ref{lem:lambda d}.
By \cite[Proposition 1.3]{PST} (in the critical case) and \cite[Proposition 2.4]{ST2} (in the subcritical case),  $u, v\in C^{2, \zeta}(\overline \Omega)$ for some $\zeta \in (0, 1)$, therefore $f, g \in L^\infty(\Omega)$ and, in particular, $(f, g)\in X_n$ for every $n$. 
Thus one has
\[ D_n \ge  \frac{\int_\Omega f Kg }{\norm{f}_{\alpha_n} \norm{g}_{\beta_n}}. \]
By the reverse Fatou's lemma,  
\[ \liminf_{n \to \infty} D_n \ge \liminf_{n \to \infty} \frac{\int_\Omega f Kg }{\norm{f}_{\alpha_n}\norm{g}_{\beta_n}} \ge  \frac{\int_\Omega f K g}{\norm{f}_\alpha \norm{g}_\beta} =D_{p, q}, \]
which is the conclusion of the first step.

\noindent \textit{Step 2.} Take $(f_n, g_n) \in X_n$ achieving $D_n$ and such that $\norm{f_n}_{\alpha_n}= \norm{g_n}_{\beta_n}=1$.  We show that this is a maximizing sequence for $D_{p, q}$, after which the result follows directly from Lemma~\ref{lemma maximizing sequences}.

\noindent \textit{Step 2.1.}
We first assume $(p, q)$ is subcritical, namely we assume \eqref{CH}. 
Take $u_n:=|f_n|^{1/p_n-1} f_n$ and $v_n:=|g_n|^{1/q_n-1} g_n$, which satisfy (recall Lemma~\ref{lem:lambda d})
$-\Delta u_n = D_n^{-1} |v_n|^{q_n-1} v_n$ and $-\Delta v_n = D_n^{-1} |u_n|^{p_n-1} u_n$.  Thus, weakly we have that $\Delta(|\Delta u_n|^{\frac1{q_n}-1} \Delta u_n) = D_n^{-1-\frac1{q_n}} |u_n|^{p_n-1} u_n$
and
\begin{align*} 
&\int_\Omega |\Delta u_n|^{\frac{q_n+1}{q_n}} = D_n^{-1-\frac1{q_n}} \int_\Omega |u_n|^{p_n+1} = D_n^{-1-\frac1{q_n}} \int_\Omega |f_n|^{\frac{p_n+1}{p_n}} \le C,\\
&\int_\Omega |u_n|\leq |\Omega|^\frac{p_n}{p_n+1}\|u_n\|_{p_n+1} =|\Omega|^\frac{p_n}{p_n+1}\|f_n\|_{\frac{p_n+1}{p_n}}^\frac{1}{p_n}=|\Omega|^\frac{p_n}{p_n+1}\leq C.
\end{align*}
Notice the estimates above are uniform in $n$, since $D_n$ is bounded from below by Step 1 and $\norm{f_n}_{\frac{p_n+1}{p_n}}=1$. 
Therefore, by Corollary~\ref{Coro:equivalentnorms}, $(u_n)$ is uniformly bounded in $W^{2,\frac{q_n+1}{q_n}}(\Omega)\hookrightarrow L^{q_n^*+1}(\Omega)$, where $q_n^*$ is such that $(q_n^*, q_n)$ belongs to the critical hyperbola, namely
\[\frac{1}{q_n^*+1} +\frac{1}{q_n+1} =  \frac{N-2}{N}. \]
Similarly, one gets that $(v_n)$ is uniformly bounded in $W^{2,\frac{p_n+1}{p_n}}(\Omega)\hookrightarrow L^{p_n^*+1}(\Omega)$, where 
\[\frac{1}{p_n^*+1} +\frac{1}{p_n+1} =  \frac{N-2}{N}. \]
Since $(p, q)$ is subcritical,  $p < q_n^*$ and $q < p_n^*$, for any $n$ large enough. This implies that there exists a sufficiently small $\delta>0$ such that $u_n$ is uniformly bounded in $L^{p+1+ \delta}(\Omega)$, and $v_n$ is uniformly bounded in $L^{q+1+ \delta}(\Omega)$. 
Therefore, 
\[ \int_\Omega |f_n|^{\alpha} = \int_\Omega |u_n|^{(p+1) \frac{p_n}{p}} \le C.\]
Similarly, we get a uniform bound on $g_n$, and $(f_n, g_n) \in X$.

Moreover, 
\[ \lim_n \norm{f_n}_{\alpha}^{\alpha} = \lim_n \norm{u_n}_{p+1}^{p+1} = \lim_n \norm{u_n}_{p_n+1}^{p_n+1} = \lim_n \norm{f_n}_{\frac{p_n+1}{p_n}}^{\frac{p_n+1}{p_n}}=1, 
\]
where the second equality holds since
\[ \abs{\norm{u_n}_{p+1}^{p+1} - \norm{u_n}_{p_n+1}^{p_n+1}} \le c |p_n-p| \int_\Omega \left( |u_n|^{p+1-\delta} + |u_n|^{p+1+\delta} \right) =o(1) \]
as $n \to \infty$,
see also the proof of \cite[Equation 3.22]{ST}. Similarly, recalling that $v_n$ is uniformly bounded in $L^{q+1+ \delta}(\Omega)$, one has
\begin{equation}\label{equalities limit conv} \lim_n \norm{g_n}_{\beta}=1. 
\end{equation}

Hence, 
\[ \limsup_{n \to \infty} D_n = \limsup_{n \to \infty} \int_\Omega f_n K g_n =\limsup_{n \to \infty} \frac{\int_\Omega f_n K g_n}{\norm{f_n}_\alpha \norm{g_n}_\beta}  \le D_{p, q}. \]
In particular, $\lim_{n\to \infty} D_n =D_{p, q}$, and $(f_n, g_n)$ is a maximizing sequence for $D_{p, q}$.

\textit{Step 2.2.} 
We now assume $(p, q)$ satisfies \eqref{CH2}. 
Recalling that $(p_n, q_n) \in \mathcal{A}_+$, three situations may occur: $p_n \equiv p$ and $q_n \equiv q$, and in this case estimates are trivial; up to a subsequence $p_n <p $; or up to a subsequence $q_n < q$. 

We assume without loss of generality that $p_n < p$. 
Thus $\alpha < \alpha_n$ and $L^{\alpha_n}(\Omega) \subseteq L^{\alpha}(\Omega)$. This proves that 
$f_n \in L^{\alpha}(\Omega)$. Moreover, if $n$ is large enough,
\[ K_{q_n} f_n  \in W^{2, \alpha_n}(\Omega)\subseteq W^{2, \alpha+\delta'}(\Omega) \hookrightarrow L^{q+1+\delta}(\Omega)\]
for sufficiently small $\delta,\delta'\in(0,1)$; therefore $v_n \in L^{q+1+ \delta }(\Omega)$ and 
\[ \int_\Omega |g_n|^\frac{q+1}{q}  = \int_\Omega |g_n|^{\frac{q+1}{q_n}\frac{q_n}{q}} \le  \int_\Omega |v_n|^{q+1+\delta} \le C, \]
and $g_n \in L^{\beta}(\Omega)$.
Moreover, by Hölder's inequality, 
\[
\left\|f_{p_n}\right\|_\frac{p+1}{p}\leqslant\left\|f_{p_n}\right\|_{\frac{p_n+1}{p_n}}|\Omega|^{\frac{\left(p-p_n\right)}{(p+1)\left(p_n+1\right)}} =|\Omega|^{\frac{\left(p-p_n\right)}{(p+1)\left(p_n+1\right)}}=1+o(1). 
\]

On the other hand, since $v_n \in L^{q+1+ \delta }(\Omega)$ for some $\delta >0$, one has that \eqref{equalities limit conv} holds.  
We can now conclude that 
\[ \lim_n D_n \le \lim_n\frac{D_n}{\norm{f_n}_{\alpha} \norm{g_n}_{\beta}} = \lim_n \frac{\int_\Omega f_n K g_n}{\norm{f_n}_{\alpha} \norm{g_n}_{\beta}} \le D_{p, q}, \]
which yields $\lim_{n\to \infty} D_n \to D_{p, q}$ and, again, $(f_n,g_n)$ is a maximizing sequence for $D_{p,q}$.
\end{proof}

\begin{remark}
We point out that as a particular case of Proposition~\ref{D cont} above (for $p_n=q_n$)  we get Lemma 3.4 in \cite{ST}, see \cite[Lemma 2.6]{ST2}. 
\end{remark}
\begin{remark}\label{rmk:conditions}
Notice that in the proof above the hypothesis \eqref{CH2} was only used to guarantee  that 
\begin{equation}\label{compact} D_{p, q} > \frac{2^{2/N}}{S_{p,q}}, \end{equation} 
under which we have compactness (recall Lemma~\ref{lemma maximizing sequences}). Therefore, the conclusion of Theorem~\ref{main thm lambda} is actually valid for any $(p, q)$ such that \eqref{compact} holds.  
\end{remark}

We are now ready to prove the convergence result in Theorem~\ref{main thm lambda} for $(p, q) \in \mathcal{A}_+$. 
\begin{proof}[Proof of convergence in Theorem~\ref{main thm lambda}, case $(p, q) \in \mathcal{A}_+$]
Let us assume $(p, q) \in \mathcal{A}_+$. Given a sequence $(p_n, q_n) \to (p, q)$ we can assume without loss of generality that $p_n >0$, thus, by Lemma~\ref{lem:lambda d} and Proposition~\ref{D cont}, $\Lambda_{p_n, q_n} \to \Lambda_{p, q}$.
Moreover, if $u_{p_n, q_n}$  is an optimizer for   $\Lambda_{p_n, q_n}$,  then the couple 
\[ (f_n, g_n):= (|u_{p_n, q_n}|^{p_n -1} u_{p_n, q_n}, \Lambda_{p_n, q_n}^{-1} (-\Delta) u_{p_n, q_n}) \] 
is an optimizer for $D_{p_n, q_n}$ and strongly converges to $(f, g)$ optimizer of $D_{p, q}$ in $L^{\alpha}(\Omega) \times L^{\beta}(\Omega)$. This yields the existence of $u_{p, q}$ optimizer for $\Lambda_{p, q}$ such that 
\[ |u_{p_n, q_n}|^{p_n -1} u_{p_n, q_n} \to |u_{p, q}|^{p-1} u_{p, q} \quad \text{ strongly in } L^{\alpha}(\Omega). \]
Moreover, 
$\Delta u_{p_n, q_n} \to \Delta u_{p, q}$ in $L^{\beta}$. 

From the converse of the dominated convergence theorem, up to a subsequence $|f_n| \le h \in L^{\alpha}$, thus by dominated convergence theorem
\begin{equation}\label{wn} u_{p_n, q_n} \to u_{p, q} \text{ strongly in } L^{p+1}(\Omega). \end{equation}
It remains to prove that this convergence is actually in $C^{2, \zeta}(\overline \Omega)$ for some $\zeta$.
In the subcritical case, this follows by a bootstrap argument and  Lemma~\ref{lemma:regularity}. 
As for the critical case, we have to adapt the regularity result in \cite{PST}, below we give a sketch. Notice that in this case $pq>1$. For simplicity, we define $u_n := u_{p_n, q_n}$.

Take $r > \frac{N}{N-2}, \frac{N}{Nq_n-2q_n-2}$ for any $n$ large enough, fixed but arbitrary. We denote $G$ the Green function for 
\[ -\Delta u = \eta \text{ in } \Omega, \quad \partial_\nu u=0 \text{ on } \partial \Omega, \]
where $\eta$ is such that $\int_\Omega \eta=0$. Thus, 
\[ |G(x, y)| \le C(\Omega) |x-y|^{2-N}, \]
see \cite[Proposition 9]{DRW} and \cite[Lemma 3.1]{RW}. 
Let $\eta \in L^r(\Omega)$, and define the operator $T_n^L: L^r(\Omega) \to L^r(\Omega)$ such that 
\[ T_n^L(\eta)(x):=\int_\Omega G(x, y) \left \{ \abs{\int_\Omega G(y, z) h(z) \, dz}^{q_n-1} \left( \int_\Omega G(y, z) h(z) \, dz \right) \right \}\, dy, \]
where
\[ h(z):= |u_n^L(z)|^{p_n-\frac 1 {q_n}}  |\eta(z)|^{ \frac 1 {q_n}-1} \eta(z), \quad \text{ and }
\quad u_n^L(x):=\begin{cases} u_n(x) &\text{ if } |u_n(x)| > L, \\
0 &\text{ if } |u_n(x)| \le L. \end{cases} \]
One can prove then that
\[ \norm{T_n^L \eta }_r \le C \norm{u_n^L}_{p_n+1}^{p_nq_n-1} \norm{\eta}_r, \]
see \cite{PST}.
Notice that the constant $C$ comes from an application of the Hardy--Littlewood--Sobolev inequality, and does not depend on $n$ and $L$. 

Note that $u_n \in L^{p_n+1}(\Omega)$ is uniformly bounded. Indeed, assume $p_n \le p$ for any $n$ large enough. Then, $\norm{u_n}_{p_n+1} \le C \norm{u_n}_{p+1} \le C$. On the other hand, since $(p, q)$ is critical, if $p_n > p$ definitely, then $q_n \le q$  for any $n$ large enough. 
Furthermore, $v_n:= -\Delta u_n \in W^{2, \frac{p_n +1}{p_n}}(\Omega)$ is such that 
$v_n \to v:=\Lambda_{p, q} g$ strongly in $L^{\beta}(\Omega)$. By the inverse of the dominated convergence theorem, $|v_n| \le h \in L^{\beta}(\Omega)$ hence $|v_n|^{1/q_n-1} v_n \to |v|^{1/q-1} v $
strongly in $L^{q+1}(\Omega)$. 
Moreover, 
\[ \norm{u_n}_{p_n+1}^{\frac{q_n(p_n+1)}{q_n+1}} = \Lambda_{p_n, q_n}^{-1} \norm{\Delta u_n}_{\frac{q_n+1}{q_n}}    \le C \norm{|v_n|^{1/q_n-1} v_n}_{q_n+1}^{q_n} \le C \norm{|v_n|^{1/q_n-1} v_n}_{q+1}^{q_n} \le C. \]
Then, we can choose $L$ large enough \emph{independent of $n$} such that $\norm{u_n^L}_{p_n+1}$ is small enough.
This implies that $T_n^L$ is a contraction from $L^r$ into itself. Moreover,
\[ u_n^L =T_n^L(u_n^L) + F_n, \]
where, if $|u_n| >L,$ 
\begin{multline*} F_n:= - \int_\Omega G(x, y) \left\{ \abs{\int_\Omega G(y, z) |u_n^L(z)|^{p_n-1} u_n^L(z) \,dz}^{q_n-1} \left( \int_\Omega G(y, z) |u_n^L(z)|^{p_n-1} u_n^L(z) \,dz \right) \right\} \, dy 
\\
+ \int_\Omega G(x, y) \left\{ \abs{\int_\Omega G(y, z) |u_n(z)|^{p_n-1} u_n(z) \,dz}^{q_n-1} \left( \int_\Omega G(y, z) |u_n(z)|^{p_n-1} u_n(z) \,dz \right) \right\} \, dy,   \end{multline*}
whereas, if $|u_n| \le L$,
\[ F_n:= -T_n^L(u_n^L). \]
In both cases, $F_n$ can be estimated by a constant which only depends on $L$. Hence, we proved that 
\[ u_n^L=T_n^L(u_n^L) + F_n \quad \text{ with } \norm{F_n}_r \le C. \]
By the contraction mapping theorem, $u_n^L \in L^r(\Omega)$ and it is uniformly bounded in the $L^r(\Omega)$ norm. Therefore, the same holds for $u_n$.  
By classical regularity arguments combined with Lemma~\ref{lemma:regularity}, we conclude that $u_n$ and $v_n:=-\Delta u_n$ are uniformly bounded in the $C^{2, \zeta}(\overline \Omega)$ norm for some $\zeta \in (0, 1)$, hence $u_n \to u_{p, q}$ and $\Delta u_n \to \Delta u$ in $C^{2, \zeta}(\overline \Omega)$. 
\end{proof}

\subsection{Convergence for $(p, q) \in \mathcal{A} \setminus \mathcal{A}_+$}
\begin{lemma}\label{thm:conv:p0}
Let $(p_n, q_n) \in \mathcal{A}$, and assume $p_n \to 0$, $q_n \to q>0$. Let $u_{p_n, q_n}$ be an optimizer of $\Lambda_{p_n, q_n}$ such that $\|u_{p_n,q_n}\|_{p_n+1}=1$. Then $\Lambda_{p_n, q_n} \to \Lambda_{0, q}$, and $u_{p_n, q_n} \to u_{0}$ in $C^{2, \zeta}(\overline \Omega)$ for $\zeta \in (0, \min\{q, 1\})$, where $u_{0}$ is an optimizer of $\Lambda_{0, q}$.
\end{lemma}

\begin{proof}
Let $u_0 \in \mathcal{M}_q$ be a minimizer for $\Lambda_{0, q}$. By Proposition~\ref{prop:existence_Lambda}, we have $u_0 \in C^{2,\zeta}(\overline \Omega)$ for a suitable $\zeta >0$, and in particular $u_0$ and $\Delta u_0$ belong to $L^t(\Omega)$ for any $t \ge 1$.

Recall that there is a bounded sequence $\kappa_n:=\kappa_{p_n}(u_0) \in \R$ such that $\int_\Omega |u_0+ \kappa_n|^{p_n -1} (u_0+ \kappa_n)=0$ due to \cite[Lemma 2.1]{PW}. 
By Lemma~\ref{lem:properties M} and \cite[Lemma 3.5]{ST},
\[ \int_\Omega |u_0| \le \int_\Omega |u_0 + \kappa_n| \le |\Omega|^{\frac{p_n}{p_n+1}} \norm{u_0 + \kappa_n}_{p_n+1} = |\Omega|^{\frac{p_n}{p_n+1}} \inf_{c \in \R} \norm{u_0 + c}_{p_n+1} \le |\Omega|^{\frac{p_n}{p_n+1}} \norm{u_0}_{p_n+1}. \]
By passing to the limit and using dominated convergence
\[ \int_\Omega |u_0| \le \lim_n \norm{u_0 + \kappa_n}_{p_n+1} \le \int_\Omega |u_0|. \]
Hence in particular 
\begin{equation}\label{kn to 0} \kappa_n\to 0, \text{ as } n \to \infty, \end{equation} again by Lemma~\ref{lem:properties M}. Thus one has
\[ 
\limsup_{n \to \infty} \Lambda_{p_n, q_n} \le \limsup_{n \to \infty} \frac{\norm{ \Delta u_0}_{\frac{q_n+1}{q_n}}}{\norm{u_0+\kappa_n}_{p_n+ 1}} = \frac{\norm{\Delta u_0}_{\frac{q+1}{q}}}{ \norm{u_0}_1} = \Lambda_{0, q}. 
\]

\smallbreak

Let us now prove that $\Lambda_{0, q} \le \liminf_{n \to \infty} \Lambda_{p_n, q_n}$.
Observe that there exists a small $\delta >0$ such that 
\[ \norm{\Delta u_n}_{\frac{q+1}{q}-\delta }  \le C\norm{\Delta u_n}_{\frac{q_n+1}{q_n}} = C\Lambda_{p_n, q_n} \le D, \]
and by Lemma~\ref{lemma:regularity}, $u_n$ is uniformly bounded in $W_\nu^{2, \beta-\delta}(\Omega)$, which implies 
that there exists $\hat u \in W^{2, \beta-\delta}(\Omega)$ such that up to a subsequence 
$u_n \to \hat u$ in $L^1(\Omega)$. By the inverse of the dominated convergence theorem, there exists $U \in L^1(\Omega)$ such that $|u_n| \le |U|$. 
Let $v_n:=|\Delta u_n|^{\frac 1 {q_n}-1} \Delta u_n \in C^{2, \zeta}(\overline \Omega)$ by Proposition~\ref{equivalence}. 
Hence, for each fixed $t\geq 1$ and for any $n$ large enough so that $p_nt\leq 1$,
\begin{equation}\label{eq:anotherestimate}  |\Delta v_n|^t= \Lambda_{p_n, q_n}^{\frac{q_n+1}{q_n} t} |u_n|^{p_nt} \le C |U|^{p_n t} \le C (1+|U|). 
\end{equation}
Therefore, since $\|v_n\|_{q_n+1}$ is bounded, by Lemma~\ref{lemma:regularity}, $v_n$ is uniformly bounded in $W^{2, t}(\Omega)$ for any $t \ge 1$, thus in $C^{1, \zeta}(\overline \Omega)$ for a suitable $\zeta >0$. Applying again Lemma~\ref{lemma:regularity}, and recalling that $\norm{u_n}_1$ is uniformly bounded, we conclude that $u_n$ is uniformly bounded in $C^{2, \zeta}(\overline \Omega)$.  As a consequence, there exists $\hat u \in C^{2, \zeta}(\overline \Omega)$ such that $u_n \to \hat u $ in $C^{2, \zeta}(\overline \Omega)$. 
Notice that the estimates above also prove that there exists $\hat v$ such that $v_n \to \hat v$ in $C^{1, \zeta}(\overline \Omega)$. Fruthermore, since $\norm{\hat u }_1=\lim_n \norm{u_n}_{p_n+1}=1$, one has $\hat u \ne 0$. 

Now we show that $\hat{u}\in \mathcal{M}_q$, after which we can conclude. Notice that  
\begin{equation}\label{limit hat u M} 0= \lim_{n \to \infty} \int_\Omega |u_n|^{p_n-1} u_n =\lim_{n \to \infty} \int_{\hat u=0}|u_n|^{p_n-1} u_n+ \int_\Omega \sign(\hat u) 
\end{equation}
Recall that $\Delta v_n \to \Delta \hat v$ a.e.; by \eqref{eq:anotherestimate}  and dominated convergence, we have that
\[ \int_{\{\hat u= 0\}} |\Delta v_n| \to \int_{\{\hat u= 0\}} |\Delta \hat v|. \]
Moreover, for a.e. $x\in \{\hat u=0\}$,  $\Delta \hat u =0$ (by applying twice by \cite[Lemma 7.7]{GT}), thus $0=-\Delta \hat u = |\hat v |^{q-1} \hat v$ a.e. in $\{\hat u=0\}$. Therefore, also $\Delta \hat v=0$ (using once again \cite[Lemma 7.7]{GT}). 
Since $-\Delta v_n = \Lambda_{p_n, q_n}^{\frac{q_n+1}{q_n}} |u_n|^{p_n-1} u_n$ a.e., one has
\[ \lim_n \Lambda_{p_n, q_n}^{\frac{q_n+1}{q_n}} \int_{\{ \hat u =0 \}}  |u_n|^{p_n} = \lim_n \int_{\{ \hat u =0 \}} |\Delta v_n| =0. \]

Let us first assume $\Lambda_{p_n, q_n} \ge C >0$. 
Then $\lim_{n \to \infty} \int_{\hat u=0}|u_n|^{p_n-1} u_n=0$ and 
by \eqref{limit hat u M} one has $|\{ \hat u >0 \}|-|\{\hat u < 0 \}|=0$ and   $\hat u \in \mathcal{M}_q$. 

Let us now consider the case $\Lambda_{p_n, q_n} \to 0$. However, 
\[ 0= \lim_n \Lambda_{p_n, q_n} = \lim_n \norm{v_n}_{q_n+1} = \norm{\hat v}_{q+1}. \]
This implies $\hat v=0$ in $\Omega$, and $-\Delta \hat u = |\hat v|^{q-1} \hat v =0$ a.e. 
By $C^{2, \zeta}(\overline \Omega)$ convergence of $u_n$ to $\hat u$, one has $\partial_\nu \hat u=0$ on $\partial \Omega$. This implies by Lemma~\ref{lemma:regularity} that $\hat u$ is a.e. constant and different from $0$. Thus, $|\{ \hat u=0\}|=0$, and \eqref{limit hat u M} implies once again $\hat u \in \mathcal{M}_q$.

Therefore,
\[ \liminf_{n \to \infty} \Lambda_{p_n, q_n} = \liminf_{n \to \infty} \frac{\|\Delta u_n\|_{\frac{q_n+1}{q_n}}}{ \norm{u_n}_{p_n+1}} \ge \frac{\|\Delta \hat u\|_{\frac{q+1}{q}}}{ \norm{\hat u}_1} \ge \Lambda_{0, q}. 
\]
This yields $\Lambda_{p_n, q_n} \to \Lambda_{0, q}$. Also,
\[
\Lambda_{0, q} = \frac{\norm{\Delta \hat u}_{\frac{q+1}{q}}}{ \norm{\hat u}_1},  
\]
namely $\hat u$ is a minimizer for $\Lambda_{0, q}$.
\end{proof}

\section{Convergence results for $c_{p, q}$}\label{sec:conv:2}

\subsection{Proof of Theorem~\ref{prop:convergence direct}}
We first recall the following 
\begin{prop}[Proposition 2.5 in \cite{PST}]\label{prop:equiv}
Let $(p, q) \in \mathcal{A}_+ \setminus \mathcal{H}$, and assume 
$D_{p, q}$ is attained. Then $(u, v)=((D_{p,q})^{-q \frac{p+1}{pq-1}} K_p g, (D_{p,q})^{-p \frac{q+1}{pq-1}} K_q f)$ is a least energy solution for~\eqref{system} if and only if $(f, g)$ is a optimizer for $D_{p, q}$.  Moreover, 
\begin{equation} \label{eq:relation_energylevels}
D_{p,q}^{-\frac{(p+1)(q+1)}{pq-1}}=\frac{(p+1)(q+1)}{pq-1} c_{p, q}. 
\end{equation}
\end{prop}
Now, combining Proposition~\ref{prop:equiv} and Lemma~\ref{lem:lambda d}, one immediately gets Lemma~\ref{lemma lambda c} for $(p,q)\in \mathcal{A}_+\setminus \mathcal{H}$. 
\begin{proof}[Proof of Lemma~\ref{lemma lambda c}, case $(p,q)\in \mathcal{A}_+\setminus \mathcal{H}$]
Notice that if $(u, v)$ is a optimizer for $c_{p, q}$, then 
\[ c_{p, q}= \frac{p q-1}{(p+1)(q+1)} \int_\Omega |u|^{p+1},\quad \text{ which implies } \quad \norm{u}_{p+1}=\Lambda_{p, q}^{\frac{q+1}{p q-1}}.  \]
Thus, $u= \norm{u}_{p+1} w=\Lambda_{p, q}^{\frac{q+1}{p q-1}} w$, with $w$ such that $\norm{w}_{p+1}=1$. 
By Lemma~\ref{lem:lambda d},  $w$ is a optimizer for $\Lambda_{p,q}$ if and only if 
\begin{align*} (f, g):=(|w|^{p-1}w, \Lambda_{p, q}^{-1} (-\Delta w)) &= \left( \frac{|u|^{p-1} u}{\norm{u}_{p+1}^p}, -D_{p, q} \frac{\Delta u}{\norm{u}_{p+1}} \right) \\
&= ((D_{p, q})^{p \frac{q+1}{pq-1}} |u|^{p-1} u, - (D_{p, q})^{\frac{q+1}{pq-1} + 1} \Delta u) \end{align*} is a optimizer for $D_{p,q }$. 
However, 
\[ ((D_{p, q})^{-q \frac{p+1}{pq-1}} K_p g, (D_{p, q})^{-p\frac{q+1}{pq-1}} K_q f) =(u, v), \]
thus, by Proposition~\ref{prop:equiv} $(f, g)$ is a optimizer for $D_{p, q}$.

The  converse implication follows similarly. 
\end{proof}

For $p=0$, we need to prove the equivalence in a direct way, as $D_{p, q}$ is not defined in that case. The following is the statement of Lemma~\ref{lemma lambda c} in the case $p=0$.
\begin{lemma}\label{lambda c 0}
    Let $(p, q)$ satisfy \eqref{p=0}. Then 
    \[ \Lambda_{0, q}^{-(q+1)} =-(q+1) c_{0, q}. \]
Also, $u$ is a optimizer for $\Lambda_{0, q}$ if, and only if, $(\Lambda_{0, q}^{-(q+1)} u , \Lambda_{0, q}^{-1} v)$ is an optimizer for $c_{0, q}$, where   $v:=-\Lambda_{0,q}^{-\frac 1 q} |\Delta u|^{\frac 1 q- 1} \Delta u$.
\end{lemma}
\begin{proof}
Let $w\in \mathcal{M}_q$ be an optimizer for $\Lambda_{0, q}$ with $\norm{w}_1=1$, which exists by Theorem~\ref{main thm lambda}, and satisfies \eqref{eq lambda_0}. By Proposition~\ref{equivalence}, the pair $(\tilde w,\tilde z):=(\Lambda_{0, q}^{-(q+1)} w,-\Lambda_{0,q}^{-\frac{q+1}{q}}|\Delta w|^{\frac{1}{q}-1}\Delta w)$ is a strong solution of
\[ -\Delta \tilde w = |\tilde z|^{q-1} \tilde z, \; -\Delta \tilde z = \sign(\tilde w) \text{ in } \Omega, \quad \partial_\nu \tilde w =\partial_\nu \tilde z=0 \text{ on } \partial \Omega. 
\]
In particular, it is a competitor for $c_{0, q}$, namely
\begin{align*} c_{0, q} \le I(\tilde w, \tilde z) &= \int_\Omega \nabla \tilde w  \nabla \tilde z  - \int_\Omega |\tilde w| - \frac{1}{q+1} \int_\Omega |\tilde z|^{q+1} \\
& = \frac{q}{q+1} \Lambda_{0, q}^{-(q+1) } \int_\Omega |w| - \Lambda_{0, q}^{-(q+1) } \int_\Omega |w|= -\frac{1}{q+1} \Lambda_{0, q}^{-(q+1) }. 
 \end{align*}
On the other hand, taking $(u,v)$ any strong solution of \eqref{system:p=0}, then $u\in \mathcal{M}_q$, $\Delta(|\Delta u|^{\frac{1}{q}-1}\Delta u)=\sign u$ and $\|\Delta u\|_\frac{q+1}{q}^\frac{q+1}{q}=\|u\|_1$, therefore:
\begin{align*}  
I(u,v)&=\int_\Omega \nabla u \nabla v-\int_\Omega |u|-\frac{1}{q+1}\int_\Omega |v|^{q+1}=\frac{q}{q+1}\int_\Omega |v|^{q+1}-\int_\Omega |u|\\
&\frac{q}{q+1}\int_\Omega |\Delta u|^\frac{q+1}{q}-\int_\Omega |u|=-\frac{1}{q+1}\int_\Omega |\Delta u|^\frac{q+1}{q}=-\frac{1}{q+1}\Lambda_{0,q}^{-(q+1)},
 \end{align*}
 hence $c_{0,q}\geq -\frac{1}{q+1}\Lambda_{0,q}^{-(q+1)}$, which yields the conclusion.
\end{proof}

\begin{remark}
    We point out that the proof above works, with some small adjustments, for any $(p, q) \in \mathcal{A} \setminus \mathcal{H}$. 
\end{remark}
\begin{proof}[Proof of Theorem~\ref{prop:convergence direct}]
Let $(u_n, v_n)$ be a sequence of least energy solutions for $c_{p_n, q_n}$, and assume $(p, q) \in \mathcal{A} \setminus \mathcal{H}$. 

By Lemma~\ref{lemma lambda c}
there exists $w_n$ optimizer for $\Lambda_{p_n,q_n}$ such that 
\[ (u_n, v_n)=(\Lambda_{p_n,q_n}^{\frac{q_n+1}{p_nq_n-1}} w_n, \Lambda_{p_n,q_n}^{\frac{p_n+1}{p_nq_n-1}- \frac{1}{q_n}} |\Delta w_n|^{\frac 1 {q_n}}(-\Delta w_n)). \]
Notice that 
\[ c_{p_n, q_n}= \frac{p_nq_n-1}{(p_n+1)(q_n+1)} \int_\Omega |u_n|^{p_n+1}, \]
which implies
\begin{equation}\label{norma u_n} \norm{u_n}_{p_n+1}=\Lambda_{p_n,q_n}^{\frac{q_n+1}{p_nq_n-1}}. \end{equation}
Then $\norm{w_n}_{p_n+1}=1$. 
By Theorem~\ref{main thm lambda}, there exists $w$ optimizer for $\Lambda_{p, q}$ such that $ w_n \to w$ in $C^{2, \zeta}(\overline \Omega)$. In particular, $\norm{w}_{p+1}=1$. One also has  $\Lambda_{p_n, q_n} \to \Lambda_{p,q}$, therefore by \eqref{norma u_n} 
$\norm{u_n}_{p_n+1} \to \Lambda_{p, q}^{\frac{q+1}{pq-1}}$ and moreover
\[ \lim_n c_{p_n, q_n} = \lim_n \frac{p_n q_n-1}{(p_n+1)(q_n+1)} \Lambda_{p_n,q_n}^{\frac{(p_n+1)(q_n+1)}{p_nq_n-1}} = \frac{pq-1}{(p+1)(q+1)} \Lambda_{p, q}^{\frac{(p+1)(q+1)}{pq-1}} =c_{p, q}. \]
Thus
again by Lemma~\ref{lemma lambda c}, one has that
\[ (u, v):=(\Lambda_{p, q}^{\frac{q+1}{pq-1}} w, - \Lambda_{p, q}^{\frac{p+1}{pq-1}} z), \quad \text{ with } z:=\Lambda_{p, q}^{-\frac 1q} |\Delta w|^{\frac 1q-1} \Delta w  \]
is a optimizer for $c_{p,q}$. Moreover,
 \[ u_n \to \Lambda_{p, q}^{\frac{q+1}{pq-1}} w=u, \quad \text{ in } C^{2, \zeta}(\overline \Omega), \text{ with } \zeta \in (0 \min\{q, 1\}).   \]
On the other hand, $\Delta w_n \to  \Delta w$, thus
\[ v_n = - |\Delta u_n|^{\frac 1 {q_n} -1} \Delta u_n =- \norm{u_n}_{p_n+1}^{\frac 1{q_n}} |\Delta w_n|^{\frac 1 {q_n} -1} \Delta w_n \to \Lambda_{p, q}^{\frac 1 q \frac{q+1}{pq-1}}|\Delta w|^{\frac 1q -1} \Delta w =v \]
where the convergence is in $C^{2, \eta}(\overline \Omega)$ if $p>0$, or in $C^{1, \eta}(\overline \Omega)$ if $p=0$, 
for some suitable $\eta$. 

Notice that, if $p>0$,  considering $\Lambda_{q, p}$ instead of $\Lambda_{p, q}$ we conclude that $v_n$ is uniformly bounded in $C^{2,\zeta}$ for $\zeta \in (0, \min\{p, 1\})$. If $p=0$, since  $-\Delta v_n = \sign(u_n)$, one gets the convergence in  $C^{1, \zeta}$ for any $\zeta \in (0, 1)$.

The case $pq=1$ follows by taking into account Lemma~\ref{lemma lambda c}.
Indeed, let us assume $\mu_{1,q}=\Lambda_{\frac{1}{q},q} >1$. In this case, $\Lambda_{p_n, q_n} >1$ for any $n$ large enough. Let us also assume that $p_n q_n \to 1^+$. Then, Lemma~\ref{lemma lambda c} gives
\[ c_{p_n, q_n} = \frac{p_n q_n -1}{(p_n+1)(q_n+1)} \Lambda_{p_n, q_n}^{\frac{(p_n+1)(q_n+1)}{p_nq_n-1}} \to +\infty. \]
Moreover, 
\[
\norm{u_n}_\infty=\Lambda_{p_n,q_n}^{\frac{q_n+1}{p_n q_n-1}} \norm{w_n}_\infty,\quad \norm{v_n}_\infty=\Lambda_{p_n,q_n}^{\frac{q_n+1}{q_n(p_n q_n-1)}} \norm{\Delta w_n}_\infty^\frac{1}{q_n,}
\]
where $w_n$ is an optimizer for $\Lambda_{p_n,q_n}$. 

The cases $\mu_{1,q} >1$ with $p_n q_n \to 1^-$, and $\mu_{1,q} < 1$, follow in a similar way. 
\end{proof}

\subsection{Case $pq=1$, $\mu_{1,q}=1$}

Let $(u_n,v_n)$ achieve $c_{p_n,q_n}$.  By Lemma~\ref{lemma lambda c} and Theorem~\ref{main thm lambda} 
there exists $\bar u_n$ optimizer for $\Lambda_{p_n, q_n}$ such that 
\[(u_n, v_n)=(\Lambda_{p_n, q_n}^{\frac{q_n +1}{p_n q_n-1}} \bar u_n, -\Lambda_{p_n, q_n}^{\frac{p_n +1}{p_n q_n-1} - \frac 1 {q_n}} |\Delta \bar u_n|^{\frac1 {q_n} - 1} \Delta \bar u_n ), \]
and $\bar u_n \to \varphi_1$
in $C^{2, \zeta}(\overline \Omega)$, with $\varphi_1$ satisfying~\eqref{eigen} and \eqref{eigen normaliz}. 
We now claim that, under the hypotheses of Theorem ~\ref{mu1:prop}, there exists a positive constant $\mathfrak{c}>0$ such that up to a subsequence 
\begin{equation}\label{claim limit lambda} \lim_{n\to \infty}\Lambda_{p_n, q_n}^\frac{(p_n+1)(q_n+1)}{p_nq_n-1}=:\mathfrak{c}>0. \end{equation}
Thus, 
\begin{align*}
    u:= \lim_{n\to\infty} u_n &= \lim_{n\to\infty} \Lambda_{p_n, q_n}^{\frac{q_n+1}{p_nq_n-1}} \bar u_n
    = \mathfrak{c}^\frac{1}{p+1} \varphi_1,
\end{align*}
Similarly, exploiting \eqref{eigen}, we get $v_n \to \mathfrak{c}^{\frac{1}{q+1}} \psi_1$, with $\psi_1$ satisfying \eqref{eigen normaliz}. 

The proof of Theorem~\ref{mu1:prop} will be complete once we show \eqref{claim limit lambda}. In order to do so, we separately consider the cases $p \ne q$ and $p=q=1$ 
in the next two lemmas. 
\begin{lemma}\label{lem pq=1 1}
Under the assumptions of Theorem~\ref{mu1:prop}, and if $p \ne q$, then \eqref{claim limit lambda} holds. 
\end{lemma}
\begin{proof}
Assume without loss of generality that $\pfrak(0)=p>1$ (if not, then necessarily $\qfrak(0)=q>1$, and the argument we use below would be similar, since $\Lambda_{p,q}=\Lambda_{q,p}$).    Note that $e(t):=\pfrak(t)\qfrak(t)-1=te'(0)+o(t)$ as $t\to 0.$  We may assume that 
\begin{align}\label{e:hyp}
    e'(0)=\pfrak'(0)\qfrak(0)+\pfrak(0)\qfrak'(0)>0.
\end{align}
This implies that $\pfrak(t)\qfrak(t)-1$ is positive for $t$ small. If $e'(0)<0$, then we use $1-\pfrak(t)\qfrak(t)$ instead of $\pfrak(t)\qfrak(t)-1$ in the proof below.  First, we find $C>0$ such that
\begin{align}\label{L:db 1}
    \limsup_{t\to 0}\left(\Lambda_{\pfrak(t),\qfrak(t)}^\frac{\qfrak(t)+1}{\qfrak(t)}\right)^\frac{1}{\pfrak(t)\qfrak(t)-1}<C.
\end{align}

Let $\varphi \in C^{2,\zeta}(\overline \Omega)\subseteq W^{2, \frac{\qfrak(t)+1}{\qfrak(t)}}_\nu(\Omega)$ be an optimizer for $\mu_{1,q}=\Lambda_{\pfrak(0),\qfrak(0)}=1$ 
 such that
\begin{align}\label{phi:hyp}
\int_\Omega |\Delta \varphi|^{\frac{\qfrak(0)+1}{\qfrak(0)}}\, dx=\Lambda_{\pfrak(0), \qfrak(0)}^{\frac{\pfrak(0)+1}{\qfrak(0)}}=1,\quad 
\int_\Omega |\varphi|^{\mathfrak{p}(0)+1}\, dx = 1
\quad \text{ and }\quad \int_\Omega |\varphi|^{\pfrak(0)-1} \varphi\, dx=0.
\end{align}

Let $F:[0,1]\times \R\to \R$ be given by
\begin{align*}
F(t,s):= \int_\Omega |\varphi+s|^{\pfrak(t)-1} (\varphi+s)\, dx.
\end{align*}
Then $F(0,0)=0$ and $F_s(0,0)=\pfrak(0)\int_\Omega |\varphi|^{\pfrak(0)-1}\, dx>0$ (recall that $\pfrak(0)>1$).  By the implicit function theorem, there is a function $s\in C^1([0,1))$ such that $s(0)=0$ and $F(t,s(t))=0$ for $s\in[0,\delta)$ for some $\delta>0$. Note that, using a Taylor expansion,
\begin{align*}
D(t)&:=\int_\Omega |\varphi(x)+s(t)|^{\pfrak(t)+1}\, dx=1+tD'(0)+o(t),\\
N(t)&=\int_\Omega |\Delta (\varphi(x)+s(t))|^\frac{\qfrak(t)+1}{\qfrak(t)}\, dx=\int_\Omega |\Delta \varphi(x)|^\frac{\qfrak(t)+1}{\qfrak(t)}\, dx=1+tN'(0)+o(t),
\end{align*}
 as $t\to 0^+$, where
 \begin{align*}
D'(0)=\pfrak'(0)\int_\Omega |\varphi|^{\pfrak(0)+1} \ln|\varphi|\, dx,\quad
N'(0)=-\frac{\pfrak'(0)}{\qfrak(0)^2}\int_\Omega |\Delta \varphi|^\frac{\qfrak(0)+1}{\qfrak(0)}\ln |\Delta \varphi|\, dx.
\end{align*}
Here we have used~\eqref{phi:hyp}. Then, $\Phi(x,t):=(\varphi(x)+s(t))/D(t)^\frac{1}{\pfrak(t)+1}$ is a competitor and
\begin{align*}
\left(\Lambda_{\pfrak(t),\qfrak(t)}^\frac{\qfrak(t)+1}{\qfrak(t)}\right)^\frac{1}{\pfrak(t)\qfrak(t)-1}
\leq \left(\int_\Omega |\Delta \Phi(x,t)|^\frac{\qfrak(t)+1}{\qfrak(t)}\, dx\right)^\frac{1}{\pfrak(t)\qfrak(t)-1}
=\left(\frac{1+tN'(0)+o(t)}{1+tD'(0)+o(t)}\right)^\frac{1}{e'(0)t+o(t)}
\to e^\frac{N'(0)-D'(0)}{e'(0)}. 
\end{align*}
This implies~\eqref{L:db 1}.   Next, we find $C>0$ such that
\begin{align}\label{L:db} 
    \liminf_{t\to0^+}\left(\Lambda_{\pfrak(t),\qfrak(t)}^\frac{\qfrak(t)+1}{\qfrak(t)}\right)^\frac{1}{\pfrak(t)\qfrak(t)-1}>C.
\end{align}
Recall that $\pfrak(t)\qfrak(t)-1>0$ for small $t>0$, by~\eqref{e:hyp}. Now, we adapt the argument from \cite[Lemma~3.4]{ST}.  By Lemma~\ref{lem:lambda d}, we have that 
\begin{align}\label{eq:relation_aux}
    \Lambda_{\pfrak(t), \qfrak(t)}= D_{\pfrak(t), \qfrak(t)}^{-1},
\end{align}
so that \eqref{L:db} follows from an upper bound on $D_{\pfrak(t), \qfrak(t)}$.
Let $(f(t),g(t))$ be a maximizers of $D_{\pfrak(t), \qfrak(t)}$ such that $\|f(t)\|_\frac{\pfrak(t)+1}{\pfrak(t)}=1$ and $\|g(t)\|_\frac{\qfrak(t)+1}{\qfrak(t)}=1$. Then,
\begin{align*}
    \left|
    \int_\Omega |f(t)|^\frac{\pfrak(t)+1}{\pfrak(t)}-|f(t)|^\frac{\pfrak(0)+1}{\pfrak(0)}\, dx
    \right|
    &=\left|\frac{\pfrak(t)+1}{\pfrak(t)}-\frac{\pfrak(0)+1}{\pfrak(0)}\right|    \left|
    \int_\Omega \int_0^1 |f(t)|^{s\frac{\pfrak(t)+1}{\pfrak(t)}+(1-s)\frac{\pfrak(0)+1}{\pfrak(0)}}\ln|f(t)|\, ds\, dx
    \right|\\
    &\leq Ct,
\end{align*}
where we used that there are $\delta,k>0$ such that, for all $t>0$ small,
\begin{align*}
\left|\int_\Omega|f(t)|^{s\frac{\pfrak(t)+1}{\pfrak(t)}+(1-s)\frac{\pfrak(0)+1}{\pfrak(0)}}\ln|f(t)|\, dx\right|
\leq k\int_\Omega|f(t)|^{\frac{\pfrak(0)+1}{\pfrak(0)}+\delta}+|f(t)|^{\frac{\pfrak(0)+1}{\pfrak(0)}-\delta}
\, dx\leq C
\end{align*}
by the regularity of $f(t)$, arguing as in \cite[Lemma 3.4]{ST}.  As a consequence, there is $c>0$ such that
\begin{align*}
\|f(t)\|_{\frac{\pfrak(0)+1}{\pfrak(0)}} \leq 1+ct\quad \text{ as }t\to 0^+ \quad \text{ and, similarly, }\quad \|g(t)\|_{\frac{\qfrak(0)+1}{\qfrak(0)}} \leq 1+ct.
\end{align*}
Recall that we are assuming~\eqref{e:hyp}. Therefore, using that $D_{\pfrak(0),\qfrak(0)}=1$ and \eqref{equivalent_def_D},
\begin{align*}
\limsup_{t\to 0^+}D_{\pfrak(t),\qfrak(t)}^\frac{1}{\pfrak(t)\qfrak(t)-1}
&=\limsup_{t\to 0^+}\left( \int_\Omega f(t) K g(t) \right)^\frac{1}{\pfrak(t)\qfrak(t)-1}\\
&=\limsup_{t\to 0^+}
\left(\|f(t)\|_\frac{\pfrak(0)+1}{\pfrak(0)}\|g(t)\|_\frac{\qfrak(0)+1}{\qfrak(0)} \right)^\frac{1}{\pfrak(t)\qfrak(t)-1}
\left( \frac{\int_\Omega f(t) K g(t)}{\|f(t)\|_\frac{\pfrak(0)+1}{\pfrak(0)}\|g(t)\|_\frac{\qfrak(0)+1}{\qfrak(0)}} \right)^\frac{1}{\pfrak(t)\qfrak(t)-1}\\
&\leq \limsup_{t\to0^+}
(1+ct+o(t))^\frac{1}{te'(0)+o(t)}\leq C
\end{align*}
for some $C>0$. Here, we used that 
\begin{align*}
     \frac{\int_\Omega f(t) K g(t)}{\|f(t)\|_\frac{\pfrak(0)+1}{\pfrak(0)}\|g(t)\|_\frac{\qfrak(0)+1}{\qfrak(0)}} \leq D_{\pfrak(0),\qfrak(0)}=1.
\end{align*}
This, together with \eqref{eq:relation_aux}, directly implies \eqref{L:db}.

Now, let 
$p_n:=p(\tfrac{1}{n})$, $q_n:=q(\tfrac{1}{n})$, $p:=\pfrak(0)$, $q:=\qfrak(0)$. Then, by \eqref{L:db 1} and \eqref{L:db}, up to a subsequence, \eqref{claim limit lambda} holds. 
\end{proof}

We preliminary observe that, if $p=q=1$, then the couple $(\varphi_1, \psi_1)$ satisfying \eqref{eigen} and \eqref{eigen normaliz} is such that $\varphi_1=\psi_1$ (see \cite[Lemma 2.6]{ST2}) and it coincides with the normalized eigenfunction of the first nontrivial Neumann eigenvalue of the scalar equation, namely 
\begin{equation}\label{eq:Neumann_eigen} -\Delta \varphi_1 = \mu_1\varphi_1 \text{ in } \Omega, \quad \partial_\nu \varphi_1 =0 \text{ on } \partial \Omega, 
\end{equation}
and $\mu_1=\mu_{1,1}=1$. Let us denote $\Lambda_n:=\Lambda_{p_n, q_n}$.
\begin{lemma}\label{lem pq=1 2}
    Under the assumptions of Theorem~\ref{mu1:prop} with $p=q=1$, one has that  \eqref{claim limit lambda} holds. \end{lemma}
\begin{proof}
   To see this, let $(\bar u_n, \bar v_n)$ be such that
\begin{equation}\label{eq uv}
    -\Delta \bar u_n = \Lambda_n |\bar v_n|^{q_n-1}\bar v_n, \quad -\Delta \bar v_n = \Lambda_n |\bar u_n|^{p_n-1}\bar u_n, \quad \partial_\nu \bar u_n=\partial_\nu \bar v_n=0 \text{ on } \partial \Omega, 
\end{equation}
see Proposition~\ref{equivalence} and Remark~\ref{eq:scalings_remark}. 
We now test both equations in~\eqref{eq uv} with $\varphi_1$, whereas we test \eqref{eq:Neumann_eigen} with both $\bar u_n$ and $\bar v_n$. Therefore, we get
\begin{align*}
    \mu_1 \int_\Omega \varphi_1 \bar v_n &= \int_\Omega \nabla \varphi_1 \nabla \bar v_n = \Lambda_n \int_\Omega |\bar u_n|^{p_n-1}\bar u_n \varphi_1,\quad \mu_1 \int_\Omega  \varphi_1 \bar u_n &= \int_\Omega \nabla \varphi_1 \nabla \bar u_n = \Lambda_n \int_\Omega |\bar v_n|^{q_n-1}\bar v_n \varphi_1.
\end{align*}
Let us now sum up the two equalities above, and recall $\mu_1=\mu_{1,1}=1$, to obtain
\[
\Lambda_n \left(\int_\Omega |\bar u_n|^{p_n-1}\bar u_n \varphi_1 + \int_\Omega |\bar v_n|^{q_n-1}\bar v_n \varphi_1 \right) = \int_\Omega  \varphi_1 \bar v_n + \int_\Omega \varphi_1 \bar u_n,
\]
which we rewrite as 
\begin{equation}\label{sum asympt}
\int_\Omega \bar u_n \varphi_1 \left( |\bar u_n|^{p_n-1}\Lambda_n -1  \right) + \int_\Omega \bar v_n \varphi_1 \left( |\bar v_n|^{q_n-1}\Lambda_n - 1\right)=0.
\end{equation}
Consider the first integral. One has
\begin{align*} 
\int_\Omega \bar u_n \varphi_1 \left( |\bar u_n|^{p_n-1}\Lambda_n - 1 \right)&= \int_\Omega \bar u_n \varphi_1 \left( |\bar u_n|^{p_n-1} \Lambda_n - |\bar u_n|^{p_n(1-q_n)} \right) + \int_\Omega |\bar u_n|^{p_n(1-q_n)} \bar u_n \varphi_1 - \int_\Omega \bar u_n \varphi_1 \\
& = \int_\Omega |\bar u_n|^{p_n(1-q_n)} \bar u_n \varphi_1 \left( \Lambda_n |\bar u_n|^{p_nq_n-1} -1 \right) + \int_\Omega |\bar u_n|^{p_n(1-q_n)} \bar u_n \varphi_1- \int_\Omega \bar u_n \varphi_1.
\end{align*}
One has
\begin{align*}
& \int_\Omega |\bar u_n|^{p_n(1-q_n)} \bar u_n \varphi_1 \left( \Lambda_n |\bar u_n|^{p_nq_n-1} -1 \right)\\
&= \int_\Omega   |\bar u_n|^{p_n(1-q_n)} \bar u_n \varphi_1 \int_0^1 \frac{d}{ds} \left( \left( (\Lambda_n)^{\frac 1{p_nq_n-1}} |\bar u_n| \right)^{s(p_nq_n-1)} \right) \, ds \, dx\\
&= (p_nq_n-1) \int_\Omega \int_0^1 |\bar u_n|^{p_n(1-q_n)} \bar u_n \varphi_1 \left( \frac{\ln (\Lambda_n)}{p_nq_n-1}  + \ln |\bar u_n| \right) (\Lambda_n)^s |\bar u_n|^{s(p_nq_n-1)} \, ds \, dx.
\end{align*}
A similar expression holds for the second integral in~\eqref{sum asympt}. 
One concludes that
\begin{multline*} (p_nq_n-1) \int_\Omega \int_0^1 |\bar u_n|^{p_n(1-q_n)} \bar u_n \varphi_1 \left( \frac{\ln (\Lambda_n)}{p_nq_n-1}  + \ln |\bar u_n| \right) (\Lambda_n)^s |\bar u_n|^{s(p_nq_n-1)} \, ds \, dx  \\+ (p_nq_n-1) \int_\Omega \int_0^1 |\bar v_n|^{q_n(1-p_n)} \bar v_n \varphi_1 \left( \frac{\ln (\Lambda_n)}{p_nq_n-1}  + \ln |\bar v_n| \right) (\Lambda_n)^s |\bar v_n|^{s(p_nq_n-1)} \, ds \, dx = -h(p_n, q_n), \end{multline*}
where
\[ h(p_n, q_n)= \int_\Omega |\bar u_n|^{p_n(1-q_n)} \bar u_n \varphi_1 - \int_\Omega \bar u_n \varphi_1 +\int_\Omega |\bar v_n|^{q_n(1-p_n)} \bar v_n \varphi_1 - \int_\Omega \bar v_n \varphi_1, \]
and, by strong convergence of $(\bar u_n, \bar v_n) \to (\varphi_1, \varphi_1)$, see Theorem~\ref{main thm lambda}, 
\begin{equation}\label{limit lambda}
\lim_{n \to \infty} \ln \left( (\Lambda_n)^{\frac{1}{p_nq_n-1}} \right) = - \lim_{n \to \infty} \frac{2\int_\Omega |\varphi_1|^{2}\ln (|\varphi_1|)   + \frac{h(p_n,q_n)}{p_nq_n-1}}{2\int_\Omega |\varphi_1|^{2} }.
\end{equation}
We now focus on the term 
\[ \frac{h(p_n, q_n)}{p_n q_n -1} = \frac{\int_\Omega |\bar u_n|^{p_n(1-q_n)} \bar u_n \varphi_1 - \int_\Omega \bar u_n \varphi_1 +\int_\Omega |\bar v_n|^{q_n(1-p_n)} \bar v_n \varphi_1 - \int_\Omega \bar v_n \varphi_1
}{p_nq_n-1}. \]
Notice that
\begin{align*} \int_\Omega |\bar u_n|^{p_n(1-q_n)} \bar u_n \varphi_1 - \int_\Omega \bar u_n \varphi_1 &= \int_\Omega \bar u_n \varphi_1 \int_0^1 \frac{d}{ds} |\bar u_n|^{s(p_n-p_nq_n)}\\
&=(p_n-p_nq_n) \int_\Omega \bar u_n \varphi_1 \int_0^1 |\bar u_n|^{s(p_n-p_nq_n)} \ln |\bar u_n| 
\end{align*}
thus, also using the hypothesis $e'(0)\ne 0$, 
\[ \lim_{n \to \infty} \frac{\int_\Omega |\bar u_n|^{p_n(1-q_n)} \bar u_n \varphi_1 - \int_\Omega \bar u_n \varphi_1  }{p_nq_n-1} = \lim_{n \to \infty} \left(-1+ \frac{p_n-1}{p_nq_n-1}\right) \int_\Omega |\varphi_1|^{2} \ln |\varphi_1|.  \]
A similar expression holds for 
\[\int_\Omega |\bar v_n|^{q_n(1-p_n)} \bar v_n \varphi_1 - \int_\Omega \bar v_n \varphi_1.  \]
Therefore, using that $\norm{\varphi_1}_2=1$, see \eqref{eigen normaliz}, and recalling \eqref{limit lambda}, 
we get
\[ \lim_{n \to \infty} \ln \left( (\Lambda_n)^{\frac{1}{p_nq_n-1}} \right) = - \lim_{n \to \infty}
\frac{p_n+q_n-2}{2(p_nq_n-1)} \int_\Omega |\varphi_1|^{2} \ln |\varphi_1| . 
\]
Since, due to our assumption $e'(0) \ne 0$, 
\[\lim_{n \to \infty} \frac{p_n+q_n-2}{2(p_nq_n-1)} =\frac{\pfrak'(0)+\qfrak'(0)}{2e'(0)}=\frac 12, \quad \text{ then }\quad
\lim_{n \to \infty}  \ln \left( (\Lambda_n)^{\frac{1}{p_nq_n-1}} \right)  = 
 - \frac 14 \int_\Omega |\varphi_1|^{2} \ln (|\varphi_1|^2),
\]
 which concludes the proof.
\end{proof}
\begin{remark}\label{rem_hoe}
Notice that, when $p=q=1$, it is possible to extend the proof above to partially cover the case $e'(0)=0$, looking at higher order expansions. For instance, assume 
\[ e''(0)\ne 0, \quad p'(0)=q'(0)=e'(0)=0. \]
Then
\[ \lim_n \frac{p_n- 1}{p_nq_n-1}= \frac{p''(0)}{p''(0)+q''(0)}= \frac{p''(0)}{e''(0)} \]
and similarly
\[\lim_n \frac{q_n- 1}{p_nq_n-1} = \frac{q''(0)}{p''(0)+q''(0)}=\frac{q''(0)}{e''(0)}, \]
thus both limits exist finite. As a consequence,
\[ \lim_{n \to \infty} \frac{p_n+q_n -2}{2(p_n q_n-1)} = \frac{p''(0) + q''(0)}{2(p''(0) + q''(0))} = \frac 12,  \]
and
\[ \lim_{n \to \infty}  \ln \left( (\Lambda_n)^{\frac{1}{p_nq_n-1}} \right)  = 
 - \frac 14 \int_\Omega |\varphi_1|^{2} \ln (|\varphi_1|^2). \]
Similar arguments hold if for some $k$, $p'(0)=\dots=p^{(k)}(0)=0$, $q'(0)=\dots=q^{(k)}(0)=0$ and $e^{(k+1)}(0) \ne 0$. 
\end{remark}
\begin{remark}\label{formula} Following the same lines of the proof above, it is possible to relate the value of $ \mathfrak{c} := \lim_{n\to \infty} \Lambda_n^{\frac{(p_n+1)(q_n+1)}{p_n q_n-1}}$ with the normalized eigenfunctions $(\varphi_1,\psi_1)$, namely, the solutions of \eqref{eigen} normalized such that~\eqref{eigen normaliz} holds.  In particular, 
\begin{multline*} \lim_{n \to \infty} \ln \left( (\Lambda_n)^{\frac{1}{p_nq_n-1}} \right) = 
- \frac{\pfrak'(0)}{4e'(0)} \int_\Omega |\varphi_1|^{p+1} \ln |\varphi_1|^2 - \frac{\qfrak'(0)}{4e'(0)} \int_\Omega |\psi_1|^{q+1} \ln |\psi_1|^2 \\ 
+ \lim_{n \to \infty}\left( \frac{\int_\Omega u_n |\varphi_1|^{p-1}\varphi_1 - \int_\Omega u_n \varphi_1 |u_n|^{p-1}}{2(p_nq_n-1)}+ \frac{\int_\Omega v_n |\psi_1|^{q-1}\psi_1 - \int_\Omega v_n \psi_1 |v_n|^{q-1}}{2(p_nq_n-1)} \right).
\end{multline*}
If $p=q=1$, then $\varphi_1=\psi_1$, and 
\[ \lim_{n \to \infty} \ln \left( (\Lambda_n)^{\frac{1}{p_nq_n-1}} \right) = - \frac{\pfrak'(0)+ \qfrak'(0)}{4e'(0)} \int_\Omega |\varphi_1|^{2} \ln |\varphi_1|^2. \]
\end{remark}

\subsection{The case $p, q \to 0$}

Recall the definitions of  $I_0$, $c_0$, and $\mathcal{M}$ given in \eqref{i0}. The following result shows that if $p=q$, then \eqref{system} reduces to a scalar equation.
\begin{lemma}\label{p=q}
If $p=q \ge 0$ and $(u, v)$ is a classical solution to \eqref{system}, then $u=v$ and
\[ -\Delta u = |u|^{p-1} u \text{ in } \Omega, \quad \partial_\nu u=0 \text{ on } \partial \Omega. \]
\end{lemma}
\begin{proof}
Indeed, if $p>0$ this was proved in  \cite[Lemma 2.6]{ST2}.
The case $p=0$ follows by similar arguments. Let $u, v$ be a classical solution of~\eqref{system}, and test by $u, v$ both equations to get 
\[ \int_\Omega |\nabla u|^2 = \int_\Omega \sign(v) u, \quad \int_\Omega |\nabla v|^2 = \int_\Omega \sign(u) v, \quad \int_\Omega |u| = \int_\Omega \nabla u \nabla v = \int_\Omega |v|.  \]
Then
\[ \int_\Omega |\nabla u|^2 + \int_\Omega |\nabla v|^2 \le  \norm{v}_1 + \norm{u}_1 = 2 \int_\Omega \nabla u \nabla v, \]
whence $\nabla u = \nabla v$ in $\Omega$, and $u= v + c$ for some constant $c \in \R$. 
Notice that both $u$ and $v$ belong to $\mathcal{M}$ as $\int_\Omega \sign(u)=\int_\Omega \sign(v)=0$. Therefore, $c=0$ by Lemma~\ref{lem:properties M}. 
\end{proof}
As a consequence, it is natural to expect that least energy solutions for $c_{p_n,q_n}$ converge to least energy solutions of $c_0$ as $p_n, q_n \to 0$: this is the content of Theorem~\ref{p, q 0}, which we will prove in what follows.  
Let us start with a preliminary lemma, inspired by Lemma 3.1 in \cite{ST}. 
 \begin{lemma}\label{lem bound below}
  There exists $\delta>0$ such that $D_{p,q}\geq \delta$ for every $p\in [0,\infty)$ and $q>0$.
 \end{lemma}
 \begin{proof}
 Take $\psi$ to be an $L^2$ normalized eigenfunction with zero average associated to the first nonzero scalar Neumann eigenvalue of the Laplace operator, $\mu_1$. We have
 \[
 \|\psi\|_\frac{p+1}{p}\leq |\Omega|^\frac{p}{p+1}\| \psi \|_\infty\leq C\|\psi \|_\infty \text{ for every $p\in [0,\infty)$}
 \]
 and, analogously, $\|\psi\|_\frac{q+1}{q}\leq C\| \psi\|_\infty$ for every $q\in [0,\infty)$. Then,
 $D_{p,q}\geq \delta:=(\mu_1  C^2 \|\psi\|_\infty^2)^{-1}>0$.
 \end{proof}

\begin{lemma}\label{uniform bounds}
Let $(u_n, v_n)$ be a optimizer for $c_{p_n, q_n}$, where $p_n, q_n \to 0$.  
One has 
   \[ u_n, v_n \text{ are uniformly bounded in } W^{2,s} (\Omega) \text{ for any } s \ge 2. \]
\end{lemma}
\begin{proof}
    Assume first that $p_n, q_n >0$ for all $n$. Observe that for any $s \ge 1$ there exists $N$ such that for any $n \ge N$ one has  
\[ \norm{\Delta u_n}_{s}  \le |\Omega|^{\frac 1 s - \frac{q_n}{q_n+1}} \norm{\Delta u_n}_{\frac{q_n+1}{q_n}} \le C \norm{\Delta u_n}_{\frac{q_n+1}{q_n}}, \quad \norm{\Delta v_n}_{s}  \le |\Omega|^{\frac 1 s - \frac{p_n}{p_n+1}} \norm{\Delta v_n}_{\frac{p_n+1}{p_n}} \le C \norm{\Delta v_n}_{\frac{p_n+1}{p_n}}. \]
Since $(u_n, v_n)$ is a solution of \eqref{system}, then 
\[ \norm{\Delta u_n}_{\frac{q_n+1}{q_n}}^{\frac{q_n+1}{q_n}} + \norm{\Delta v_n}_{\frac{p_n+1}{p_n}}^{\frac{p_n+1}{p_n}} = \norm{v_n}_{q_n+1}^{q_n+1} + \norm{u_n}_{p_n+1}^{p_n+1} = \min_{c \in \R} \norm{v_n +c}_{q_n+1}^{q_n+1} + \min_{c \in \R} \norm{u_n +c}_{p_n+1}^{p_n+1}. \] 
By H\"older inequality, if $s \ge p_n +1$, $\norm{u_n +c}_{p_n+1}^{p_n+1}\le |\Omega|^{1-\frac{p_n+1}{s}}  \norm{u_n +c}_{s}^{p_n+1}$.  Moreover, 
\[ \min_{c \in \R} \norm{u_n +c}_{s}^{p_n+1} \le \norm{u_n +k _{s}(u_n)}_{s}^{p_n+1} \le \Lambda_{s-1, \frac 1{s-1}}^{-(p_n+1)} \norm{\Delta u_n }_s^{p_n+1}, \]
from which we conclude
\[ \min_{c \in \R} \norm{u_n +c}_{p_n+1}^{p_n+1} \le C \norm{\Delta u_n }_s^{p_n+1} \le C \norm{\Delta u_n}_{\frac{q_n+1}{q_n}}^{p_n+1}. \]
for any $s \ge p_n+1$. 
Analogously,
\[ \min_{c \in \R} \norm{v_n +c}_{q_n+1}^{q_n+1} \le C \norm{\Delta v_n }_s^{q_n+1} \le C \norm{\Delta v_n}_{\frac{p_n+1}{p_n}}^{q_n+1} \]
for any $s \ge q_n+1$. 
Thus, 
\[ \norm{\Delta u_n}_{\frac{q_n+1}{q_n}}^{\frac{q_n+1}{q_n}} + \norm{\Delta v_n}_{\frac{p_n+1}{p_n}}^{\frac{p_n+1}{p_n}} \le C \left( \norm{\Delta u_n}_{\frac{q_n+1}{q_n}}^{p_n+1} + \norm{\Delta v_n}_{\frac{p_n+1}{p_n}}^{q_n+1} \right),\]
whence $\norm{\Delta u_n}_{\frac{q_n+1}{q_n}}, \norm{\Delta v_n}_{\frac{p_n+1}{p_n}} \le C$,
which gives the uniform bound on the $L^s$ norms of $\Delta u_n$ and $\Delta v_n$.
Now, to conclude, we observe that the above computations also imply that $\norm{u_n}_1$ and $\norm{v_n}_1$ are uniformly bounded, thus $u_n$ and $v_n$ are uniformly bounded in $W^{2, s}(\Omega)$ for any $s$ due to Lemma~\ref{lemma:regularity}. 

Let us now consider the case (for instance) $p_n=0$ for all $n$. Then 
\[ \int_\Omega |\Delta v_n|^s \le \int_\Omega |\sign (u_n)|^s \le 1, \]
therefore $\Delta v_n$ is uniformly bounded in $L^s(\Omega)$ for any $s \ge 1$. Hence, by similar computations as above, 
\[ \norm{\Delta u_n}_{\frac{q_n+1}{q_n}}^{\frac{q_n+1}{q_n}} \le C \norm{\Delta v_n}_s^{q_n+1} \le C', \quad \text{ whereas } \quad
\norm{\Delta u_n}_s \le C \norm{\Delta u_n}_{\frac{q_n+1}{q_n}}, \]
thus also $\Delta u_n$ is uniformly bounded in $L^s(\Omega)$. An application of Lemma~\ref{lemma:regularity} completes the proof. 
\end{proof}
\begin{proof}[Proof of Theorem~\ref{p, q 0}]
Let $u_0$ be an optimizer for $c_0$, namely a least energy solution for $c_0$, see \cite{PW}. 
It satisfies 
\[ -\Delta u_0= \sign(u_0), \quad c_0=I_0(u_0)=\frac 12 \int_\Omega |\nabla u_0|^2 -  \int_\Omega |u_0|. \]
Then, 
\[ c_0 = - \frac 12 \int_\Omega |u_0|. \]
Notice that $u_0 \in W^{2, s}(\Omega)$ for any $s$, and in $C^{1, \zeta}(\overline \Omega)$, by Lemma~\ref{lemma:regularity}. 
Let $w_n:=u_0 + \kappa_n(u_0)$, where $\kappa_n$  
is given in \cite[Lemma 2.1]{PW} so that $\int_\Omega |u_0+ \kappa_n|^{p_n -1} (u_0+ \kappa_n)=0$. Since $\partial_\nu u_0=0$, then $w_n$ is a competitor for $\Lambda_{p_n, q_n}$, thus 
\[ \Lambda_{p_n, q_n} \le \frac{\norm{\Delta w_n}_{\frac{q_n+1}{q_n}}}{\norm{w_n}_{p_n+1}} = \frac{\norm{\Delta u_0}_{\frac{q_n+1}{q_n}}}{\norm{w_n}_{p_n+1}}. \]
We now notice that 
\[ \norm{\Delta u_0}_{\frac{q_n+1}{q_n}} \le |\Omega|^{\frac{q_n}{q_n+1}} \to 1, \]
whereas by \eqref{kn to 0}, 
\[ \norm{w_n}_{p_n+1} \to \norm{u_0}_1. \]
Concluding,
\[ \limsup_n c_{p_n, q_n} = \limsup_n \frac{p_n q_n -1}{(p_n+1)(q_n+1)} \Lambda_{p_n, q_n}^{\frac{(p_n+1)(q_n+1)}{p_nq_n-1}} \le  - \int_\Omega |u_0|= 2 c_0, \]
whence
\[ \limsup_n c_{p_n, q_n} \le 2 c_0. \]

Let us now consider a optimizer $(u_n, v_n)$ for $c_{p_n, q_n}$. By Lemma~\ref{uniform bounds}, $u_n$ and $v_n$ are bounded in $W^{2, s}(\Omega)$ for any $s$, which implies that there exists $(u, v)$ such that $(u_n, v_n) \to (u, v)$ in $C^{1, \zeta}(\overline \Omega)$. 

Let us assume that $u \equiv 0$. Then, since $-\Delta v_n =|u_n|^{p_n-1} u_n$, an inspection of the proof of Lemma~\ref{uniform bounds} gives 
\[ \norm{\Delta v_n}_{s} \to 0, \quad \quad \norm{v_n}_1 \to 0, \]
and by Lemma~\ref{lemma:regularity}, $v \equiv c \in \R$. 
If $c=0$, then $I_{p_n, q_n}(u_n, v_n) \to 0$ and $c_n \to 0$. However, 
\[ \Lambda_{p_n, q_n}^{\frac{(p_n+1)(q_n+1)}{p_nq_n-1}} =\frac{(p_n+1)(q_n+1)}{p_nq_n-1} c_n \to 0 \]
would lead to a contradiction with the fact that $\Lambda_{p_n, q_n}$ is uniformly bounded due to Lemma~\ref{lem bound below}. 
If $c \ne 0$, then 
\[ 0 = \int_\Omega \nabla u \nabla \varphi= \lim_n \int_\Omega \nabla u_n \nabla \varphi= \lim_n \int_\Omega |v_n|^{q_n-1} v_n \varphi = \int_\Omega \sign(c) \varphi, \]
again a contradiction. 

Hence we can assume that $u, v \not \equiv 0$. 
We now have
\[ \int_\Omega |\nabla u|^2 = \lim_n \int_\Omega |\nabla u_n|^2 = \lim_n \int_\Omega |v_n|^{q_n-1}v_n u_n \le \lim_n \norm{v_n}_\infty^{q_n} \int_\Omega |u_n| = \int_\Omega |u|. \]
Similarly, 
\[ \int_\Omega |\nabla v|^2 \le \int_\Omega |v|. \]
On the other hand,
\[ \int_\Omega \nabla u \nabla v= \lim_n \int_\Omega \nabla u_n \nabla v_n = \int_\Omega |u| =\int_\Omega |v|. \]
Then, by similar arguments as in the proof of Lemma~\ref{p=q}, $\nabla u = \nabla v$ in $\Omega$, and $u= v + c$ for some constant $c \in \R$. 

Moreover, 
\[ \lim_n \norm{v_n}_\infty^{q_n} =1 \]
implies 
\[ \abs{\int_\Omega \nabla u \nabla \psi} = \lim_n \abs{\int_\Omega |v_n|^{q_n-1}v_n \psi} \le \lim_n \norm{v_n}_\infty^{q_n} \norm{\psi}_1= \norm{\psi}_1, \]
thus $f(u): C_c^\infty(\Omega)\to \R$ defined as
\[ f(u)\psi =\int_\Omega \nabla u \nabla \psi \]
is linear and continuous for the $L^1$ norm, and there exists $\eta \in L^\infty(\Omega)$ such that
\[\int_\Omega \nabla u \nabla \psi = \int_\Omega \eta \psi \quad \text{ for all } \psi \in L^1(\Omega).  \]
On the other hand,
\[ \int_\Omega \eta \psi= \lim_n \int_\Omega \nabla u_n \nabla \psi= \lim_n \int_\Omega |v_n|^{q_n-1}v_n \psi= \int_{\{v\ne 0\}} \sign(v) \psi + \lim_n \int_{\{ v =0\}} |v_n|^{q_n-1}v_n \psi, \]
and choosing $\psi$ such that $\text{supp}(\psi) \subset \{ v \ne 0\}$, we get $\eta=\sign(v)$ in the set $\{ v\ne 0 \}$.  Notice that in the set $\{v=0\}$ one has $\nabla u=0$ a.e., thus
\[ \int_\Omega \nabla u \nabla \psi= \int_\Omega \sign (v) \psi, \]
whence $\int_\Omega \sign(v)=0$, and $v \in \mathcal{M}$. Similarly, $u \in \mathcal{M}$, therefore by Lemma~\ref{lem:properties M}, $u=v \in \mathcal{M}$. 
One gets
\[ \lim_n c_{p_n, q_n} = \lim_n I_n (u_n, v_n)=2 I_0(u) \ge 2 c_0, \]
which yields the conclusion. 
\end{proof}

\section{Symmetry breaking: the bilaplacian case}\label{sec:symm:b}

In this section we prove Theorem~\ref{symm:break:thm}.

\begin{proof}[Proof of Theorem~\ref{symm:break:thm}.]
Let $\mathcal{M}_{rad}$ denote the radially symmetric functions in $\mathcal{M}_1$, and let
\begin{align}\label{mrad:def}
    m_{rad}:=\inf_{u\in \mathcal{M}_{rad}}\varphi(u),\quad \text{ where we recall that  }
\varphi(u) = \frac{1}{2} \int_\Omega |\Delta u |^{2} - \int_\Omega |u|.   
\end{align}
The claim of the theorem follows once we have shown that
\begin{equation}\label{objective}
m:=\inf_{u\in \mathcal{M}_1}\varphi(u)<m_{rad}.
\end{equation}
The strategy is as follows.
\begin{itemize}
\item We prove that $m_{rad}$ is achieved by a monotone decreasing function. This function has a closed formula, and this allows to calculate explicitly $m_{rad}$;
\item We construct a nonradial competitor in $\mathcal{M}_1$, which has less energy than $m_{rad}$.
We now divide the proof in three steps: in step 1 we prove the existence of a radial decreasing minimizer, and compute its unique zero; in step 2 we compute $m_{rad}$ and build a nonradial competitor in dimension $N=2$; in step 3, we do the latter in higher dimensions.
\end{itemize}

Before we proceed, we observe that, with a similar proof, all main theorems in the introduction (in particular, Theorem~\ref{prop:convergence direct}) hold in the radial setting.
\smallbreak

\noindent Step 1. Let $p_n\to 0$, and $(u_n,v_n)$ be a radial least energy solution of \eqref{system} for $p=p_n$, $q=1$ (denote by $c_{p_n,1}^{rad}$ the associated level). In particular, $u_n$ is a radially symmetric least energy solution of 
\[
    \Delta^2 u =  |u|^{p_n-1} u, \quad \partial_\nu u=\partial_\nu( \Delta u)=0 \quad \text{ on } \partial \Omega.
    \]
    Without loss of generality, assume that $u_n(0)>0$.
    By Theorem~\ref{prop:convergence direct} in the radial setting, $c_{p_n,1}^{rad}\to c_{0,1}^{rad}=m_{rad}$, and we know that $u_{n} \to u$ in $C^{2, \zeta}(B)$ for some $\zeta \in (0, 1)$, where $u\in C^{2, \zeta}(B)$ is a (nontrivial) radially symmetric least energy solution of~\eqref{bilap:sb}. By \cite[Theorem 1.1]{ST2}, we know that $u_n$ and $v_n$ are strictly monotone decreasing in the radial variable and $u_n(0)v_n(0)>0$. Then, $u$ and $v:=-\Delta u$ are monotone nonincreasing. 
    
    There are $0<a_1\leq a_2<1$ such that $u>0$ in $(0,a_1)$, $u<0$ in $(a_2,1)$, and $u=0$ in $[a_1,a_2]$.  We claim that $a_1=a_2.$  Indeed, if $a_1<a_2$, then, since $u\in C^{2, \zeta}(B)$, this implies that $u(a_1)=u'(a_1)=u''(a_1)=0$. If $v=-\Delta u>0$ in $(0,a_1)$, then, by the Hopf Lemma for the (Dirichlet) bilaplacian on balls (see, for instance, \cite[Theorem 5.7]{GGS}), we would have that $\Delta u(a_1)>0$, however $\Delta u(a_1)=u''(a_1) + \frac{N-1}{a_1} u'(a_1)=0$. Hence this is not possible and there must be some $b\in (0,a_1)$ such that $v(b)<0$. But this is incompatible with the monotonicity of $v$ and the fact that $v=-\Delta u=0$ in $(a_1,a_2)$. We reach a contradiction and then $a:=a_1=a_2$ is the only root of $u$. Then:
    \begin{align*}
0 = \lim_{n\to\infty}\int_B |u_n|^{p_n-1}u_n=\lim_{n\to\infty}\int_{B\setminus \{|x|=a\}} |u_n|^{p_n-1}u_n = \int_B \sign(u)=|\{u>0\}|-|\{u<0\}|,
\end{align*}
and
\begin{align*}
\sigma_N \frac{a^{N}}{N}
=\sigma_N \int_0^{a} r^{N-1}\, dr
=\int_{\{u>0\}}1\, dr
=|\{u>0\}|=|\{u<0\}|
=\sigma_N \int_{a}^1 r^{N-1}\, dr
=\sigma_N\left(\frac{1}{N}-\frac{a^{N}}{N}\right),
\end{align*}
where $\sigma_N=|\partial B_1|$ denotes the surface measure of the unit sphere. Then 
\begin{align*}
u(a)=0\qquad \text{with $a:=2^{-\frac{1}{N}}$.}
\end{align*}

Now we adapt the ideas in \cite[Theorem 5.7]{PW}.  We note that our setting is more complex, due to the fact that the bilaplacian is of higher-order and because the set $\mathcal{M}_1$ defined in \eqref{nehari2}
imposes some additional restrictions on the normal derivative.

\noindent Step 2. \textbf{Two-dimensional case:} Let $N=2$ and recall that $a=\frac{1}{\sqrt{2}}$.  First, we solve explicitly the equation
\begin{align}\label{v:eq}
-v''-\frac{1}{r}v'=\sign (u) & \text{ in  }(0,1),\qquad v'(0)=v'(1)=0,\qquad \int_B v = 0.
\end{align}
Here, the zero average condition follows from the first equation in ~\eqref{system:p=0} (with $q=1$), because $0=\int_{B}- \Delta u = \int_B v$.

For this, we consider an auxiliary problem 
\begin{align*}
-V''-\frac{1}{r}V'=\sign (u)\quad \text{ in  }(0,1),\qquad V'(0)=V(a)=V'(1)=0.
\end{align*}

Using that $-(V'(r) r)_r = -r(V''(r)+\frac{1}{r} V'(r))=r\, \sign (u(r))$, we have that
\[
V(r)=\begin{cases}
\int_r^a \tau^{-1} \int_0^\tau s\, ds\, d\tau = \frac{1}{8}(1 - 2 r^2), & \text{ for }0<r<a.\\
\int_a^r \tau^{-1} \int_\tau^1 s (-1)\, ds\, d\tau =\frac{r^2}{4}-\frac{\ln (r)}{2}-\frac{1}{8}-\frac{\ln (4)}{8}, & \text{ for }a<r<1
\end{cases}.
\]
This is the solution found in \cite[Proof of Theorem 5.7 (iii)]{PW}.  Now, the (unique) solution $v$ of~\eqref{v:eq} must be $v=V-\frac{1}{|B|}\int_B V.$  A direct computation shows that
$\frac{1}{|B|}\int_B V=\frac{1}{16}(3 - 4 \ln(2))$ and then
\begin{align*}
v(r)=
\begin{cases}
 \frac{1}{16} \left(-4 r^2-1+\ln (16)\right), & r<\frac{1}{\sqrt{2}},\\
 \frac{1}{16} \left(4 r^2-8 \ln (r)-5\right), & r>\frac{1}{\sqrt{2}}.
\end{cases}
\end{align*}
Next, we solve the equation
\begin{align*}
-u''-\frac{1}{r}u'=v & \text{ in  }(0,1),\qquad u'(0)=u(a)=u'(1)=0.
\end{align*}
Arguing similarly as before, we obtain that
\begin{align}\label{form:u:1}
u(r)=
\begin{cases}
 \frac{1}{256} \left(2 r^2-1\right) \left(2 r^2+3-8 \ln (2)\right), & r<\frac{1}{\sqrt{2}},\\
 \frac{1}{256} \left(-4 r^4-12 r^2+32 r^2 \ln (r)+8 \ln (r)+7+\ln (4096)\right), & r>\frac{1}{\sqrt{2}},
\end{cases}
\end{align}
as can be verified by direct computations.   Now we compute the energy of $u$, namely,
\begin{align*}
m_{rad}=\varphi(u)&= \frac{1}{2}\int_B |\Delta u |^2 - \int_B |u|
= \pi\int_0^1|v(r)|^2 r\, dr - 2\pi\int_0^1|u(r)|r\,dr\notag\\
&=\frac{\pi  (24 \log (2)-19)}{1536}
\sim -0.00483606,
\end{align*}
where $m_{rad}$ is given by~\eqref{mrad:def}.

Next, we build a nonradial competitor $U$ in $\mathcal{M}_1$ which has less energy than $u$.  With a standard abuse of notation, in polar coordinates, let
\begin{align*}
U(x)=U(r,\theta):=\frac{1}{4}\left(r-\frac{r^2}{2}\right)\cos(\theta)\qquad \text{ for }r\in (0,1),\ \theta\in (-\pi,\pi).
\end{align*}
Then, its Laplacian is given by
\begin{align*}
\Delta U = \partial_{rr}U(r,\theta)+\frac{1}{r}\partial_r U(r,\theta) + \frac{1}{r^2}\partial_{\theta \theta}U(r,\theta)
=\frac{3}{8}\cos(\theta).
\end{align*}
Thus,
\begin{align*}
\int_B |\Delta U|^2 = \frac{9}{64}\int_0^1 r\, dr\int_{-\pi}^\pi\cos(\theta)^2\, d\theta =\frac{9}{128}\pi.
\end{align*}
On the other hand,
\begin{align*}
\int_B |U| 
= \int_0^1 \int_{-\pi}^\pi \frac{1}{4}\left(r-\frac{r^2}{2}\right)|\cos(\theta)|r\, d\theta\, dr
=\frac{1}{4}\int_0^1 \left(r^2-\frac{r^3}{2}\right)\, dr\int_{-\pi}^\pi |\cos(\theta)|\, d\theta
=\frac{5}{24}.
\end{align*}
Then, if $m$ is given by~\eqref{objective}, 
\begin{align*}
m\leq\varphi(U)=\frac{1}{2}\int_B |\Delta U|^2 - \int_B |U| 
=\frac{9}{256}\pi - \frac{5}{24}\sim -0.0978867
<-0.00483606\sim \varphi(u)=m_{rad},
\end{align*}
and \eqref{objective} follows for $N=2$. 

\noindent Step 3. \textbf{Higher dimensional case:} Let $N\geq 3$ and recall that $a=2^{-\frac{1}{N}}$.   First, we solve explicitly the equation
\begin{align}\label{v:eq:n}
-v''-\frac{N-1}{r}v'=\sign (u) & \text{ in  }(0,1),\qquad v'(0)=v'(1)=0,\qquad \int_B v = 0.
\end{align}

For this, we consider an auxiliary problem 
\begin{align*}
-V''-\frac{N-1}{r}V'=\sign (u)\quad \text{ in  }(0,1),\qquad V'(0)=V(a)=V'(1)=0.
\end{align*}

Using that $-(V'(r) r^{N-1})_r = -r^{N-1}(V''(r)+\frac{1}{r} V'(r))=r^{N-1}\, \sign (u(r))$, we have that
\[
V(r)=\begin{cases}
\int_r^a \tau^{1-N} \int_0^\tau s^{N-1}\, ds\, d\tau =\frac{a^2-r^2}{2N},& \text{ for } 0<r<a,\\
\int_a^r \tau^{1-N} \int_\tau^1 s^{N-1} (-1)\, ds\, d\tau 
=\frac{1}{2 N}\left(\frac{2 r^{2-N}}{N-2}-\frac{4^{-1/N} (N+2)}{N-2}+r^2\right), & \text{ for } a<r<1.
\end{cases}
\]
This is the solution found in \cite[Proof of Theorem 5.7 (iii)]{PW}.  Now, $v=V-\frac{1}{|B|}\int_B V$ is the unique solution of~\eqref{v:eq:n}. Direct computations show that 
\begin{align*}
\frac{1}{|B|}\int_B V =
    \frac{4^{-\frac{1}{N}} \left(N \left(\left(4^{\frac{1}{N}}-1\right)
   N-1\right)-2\right)}{N\left(N^2-4\right)}
\end{align*}
and then
\begin{align*}
v(r)=
\begin{cases}
 \frac{\frac{1}{2}N \left(4^{-1/N} (3N+2)-N \left(r^2+2\right)\right)+2
   r^2}{N \left(N^2-4\right)}, & r<a,\\
 -\frac{-\left((N+2) r^{2-N}\right)-2^{-\frac{N+2}{N}} N \left(4^{\frac{1}{N}}
   N \left(r^2-2\right)+N-2\right)+2 r^2}{N \left(N^2-4\right)}, & r>a.
\end{cases}
\end{align*}
Next, we solve the equation
\begin{align*}
-u''-\frac{1}{r}u'=v & \text{ in  }(0,1),\qquad u'(0)=u(a)=u'(1)=0.
\end{align*}
We have that, for $0<r<a$,
\begin{align}\label{form:u:2}
u(r)=\int_r^a \tau^{1-N} \int_0^\tau s^{N-1}v(s)\, ds\, d\tau 
=\frac{2^{-\frac{4}{N}-3} \left(4^{\frac{1}{N}} r^2-1\right)
   \left(4^{\frac{1}{N}} \left((N-2) r^2+4 N\right)-5 N-6\right)}{N
   \left(N^2-4\right)}
,
\end{align}
and, for $a<r<1$ and $N\neq 4$,
\begin{align}\label{form:u:3}
u(r)
&=\int_a^r \tau^{1-N} \int_\tau^1 s^{N-1}v(s)\, ds\, d\tau \\
&=
\frac{-4^{\frac{N-1}{N}} (N-4) r^{2-N}
+4 (N+2) r^{4-N}
-(N-4) (N-2) r^4}{8 ((N-4)  (N-2)  N  (N+2))}\\
&+\frac{2^{\frac{N-2}{N}} (N-4)
   \left(\left(2^{\frac{N+2}{N}}-1\right) N+2\right) r^2
   +16^{-1/N} \left(3 (N-6)N
   -4^{\frac{1}{N}+1} (N-4) N-24\right)}{8 ((N-4)  (N-2)  N  (N+2))},
\end{align}
while, for for $a<r<1$ and $N=4,$
\begin{align}\label{form:u:4}
u(r)=-\frac{2 r^6+2 \left(\sqrt{2}-8\right) r^4+r^2 \left(24 \ln (r)+8 \sqrt{2}-7+\ln (64)\right)+2
   \sqrt{2}}{384 r^2}.
\end{align}

Now we compute the energy of $u$, namely, for $N\neq 4$,
\begin{align}\label{ene1}
\varphi(u)&= \frac{1}{2}\int_B |\Delta u |^2 - \int_B |u|
=\sigma_N\int_0^1|v(r)|^2 r^{N-1}\, dr - \sigma_N\int_0^1|u(r)|r^{N-1}\,dr\\
&=\frac{16^{-\frac{N+1}{N}} \left(N \left(N^2-3 N-2^{\frac{N+2}{N}} (N-4) (N+4)+16^{\frac{1}{N}}
   (N-2) (N+5)-34\right)-24\right) \pi ^{N/2}}{(N-4) (N-2) (N+2) \Gamma
   \left(\frac{N}{2}+3\right)},
\end{align}
and, for $N=4,$
\begin{align*}
\varphi(u)=-\frac{\pi ^2 \left(-57+32 \sqrt{2}+18 \ln (2)\right)}{4608}\sim -0.00156672.
\end{align*}

For $N\neq 4$, let 
\begin{align*}
    p(N)&:=N \left(N^2-3 N-2^{\frac{N+2}{N}} (N-4) (N+4)+16^{\frac{1}{N}}
   (N-2) (N+5)-34\right)-24\\
   &=\left(4^{\frac{1}{N}}-1\right)^2 N^3+3 \left(16^{\frac{1}{N}}-1\right)
   N^2+\left(2^{\frac{2}{N}+5}-5\ 2^{\frac{N+4}{N}}-34\right) N-24.
\end{align*}
If $m_{rad}$ is given by~\eqref{mrad:def}, then
\begin{align*}
&m_{rad}=\varphi(u)=\frac{16^{-\frac{N+1}{N}}p(N) \pi ^{N/2}}{(N-4) (N-2) (N+2) \Gamma
   \left(\frac{N}{2}+3\right)}=:h_1(N).
\end{align*}

Next, we build a nonradial competitor $U$ in $\mathcal{M}_1$ which has less energy.  With a standard abuse of notation, in hyperspherical  coordinates, let
\begin{align*}
U(x)=U(r,\theta):=t\left(r-\frac{r^2}{2}\right)\cos(\theta_1)
\end{align*}
for $r\in (0,1),$ $\theta_i\in (0,\pi)$ for $i=1,\ldots,N-2$, and $\theta_{N-1}\in(-\pi,\pi)$. Here $t>0$ is a parameter to be defined later.  Then, its Laplacian is given by
\begin{align*}
\Delta U 
&= \partial_{rr}U(r,\theta)+\frac{N-1}{r}\partial_r U(r,\theta) 
+ \frac{1}{r^2\sin^{N-2}(\theta)}\partial_{\theta}\Big(\sin^{N-2}(\theta)U_\theta(r,\theta)\Big)=-\frac{t}{2} (N+1) \cos (\theta ),
\end{align*}
see, for example, \cite[eq. (2.10)]{campos2020hyperspherical}. Thus,  
\begin{align*}
&\int_B |\Delta U|^2 
=\int_0^1\int_{-\pi}^\pi \int_{0}^\pi\ldots\int_{0}^\pi  |\Delta U(r,\theta_1)|^2 r^{N-1}
\sin^{N-2}(\theta_1)
\sin^{N-3}(\theta_2)\ldots \sin(\theta_{N-2})
\, d\theta_1\ldots d\theta_{N-2}d\phi\, dr
\\
&=\left(\frac{t^2}{4}(N+1)^2 \int_0^1 r^{N-1}\, dr\right)\int_{0}^{\pi}\cos^2\theta_1\sin^{N-2}(\theta_1)\, d\theta_1
|\mathbb{S}^{N-2}|\\
&=\left(\frac{t^2}{4}\frac{(N+1)^2}{N}\right)
\left( \frac{\sqrt{\pi } \Gamma \left(\frac{N-1}{2}\right)}{2 \Gamma \left(\frac{N}{2}+1\right)} \right)
\frac{2\pi^{\frac{N-1}{2}}}{\Gamma(\frac{N-1}{2})}
=\frac{t^2}{4}\frac{(N+1)^2}{N}
\frac{\pi^{\frac{N}{2}}}{\Gamma \left(\frac{N}{2}+1\right)}
= t^2\frac{(N+1)^2 \pi ^{N/2}}{2 N^2 \Gamma \left(\frac{N}{2}\right)}.
\end{align*}
On the other hand,
\begin{align*}
\int_B |U| 
&= t\int_0^1 \left(r-\frac{r^2}{2}\right)r^{N-1}\, dr\int_{0}^{\pi}|\cos(\theta_1)|\sin^{N-2}(\theta_1)\, d\theta_1 |\mathbb{S}^{N-2}|\\
&=t
\left(\frac{N+3}{2 N^2+6 N+4}\right)
\left(\frac{2}{N-1}\right)
\left(
\frac{2\pi^{\frac{N-1}{2}}}{\Gamma(\frac{N-1}{2})}
\right)
=t\frac{(N+3) \pi ^{\frac{N-1}{2}}}{2 (N+2) \Gamma \left(\frac{N+3}{2}\right)}.
\end{align*}
Then,
\begin{align*}
f(t)&=\varphi(U)=\frac{1}{2}\int_B |\Delta U|^2 - \int_B |U|=t^2\frac{(N+1)^2 \pi ^{N/2}}{4 N^2 \Gamma \left(\frac{N}{2}\right)}
-
t\frac{(N+3) \pi ^{\frac{N-1}{2}}}{2 (N+2) \Gamma \left(\frac{N+3}{2}\right)}.
\end{align*}
A direct computation shows that $f$ achieves its minimum at 
$t_0=\frac{N^2 (N+3) \Gamma \left(\frac{N}{2}\right)}{\sqrt{\pi } (N+1)^2 (N+2) \Gamma
   \left(\frac{N+3}{2}\right)}$ and 
\begin{align*}
m\leq f(t_0)=-\frac{2 N (N+3)^2 \pi ^{\frac{N}{2}-1} \Gamma \left(\frac{N}{2}+1\right)}{(N+1)^4 (N+2)^2 \Gamma\left(\frac{N+1}{2}\right)^2}=:h_2(N)<0.
\end{align*}
In particular, $h_2(4)=-\frac{3136}{50625}\sim -0.0619457<-0.00156672\sim \varphi(u),$ which settles the case $N=4.$ 

It remains to show that, for $N\geq 3$ and $N\neq 4$,
\begin{align}\label{ineq}
m\leq h_2(N)<h_1(N)=m_{rad}.
\end{align}

This can be shown using some fine estimates on Gamma functions and polynomials or using some rough estimates and computing some values by hand (or by computer). Here we follow the second alternative. Assume, for the moment, that $N\geq 5$. We want to verify that
\begin{align*}
  -\frac{2 N (N+3)^2 \pi ^{\frac{N}{2}-1} \Gamma \left(\frac{N}{2}+1\right)}{(N+1)^4 (N+2)^2 \Gamma\left(\frac{N+1}{2}\right)^2}= h_2(N)<h_1(N)=\frac{16^{-\frac{N+1}{N}}p(N) \pi ^{N/2}}{(N-4) (N-2) (N+2) \Gamma
   \left(\frac{N}{2}+3\right)},
\end{align*}
or, equivalently,
\begin{align*}
  -p(N)< q(N):=\frac{
  (N-4) (N-2) \Gamma
   \left(\frac{N}{2}+3\right)
  2 N (N+3)^2 \Gamma \left(\frac{N}{2}+1\right) 16^{\frac{N+1}{N}}}{(N+1)^4 (N+2) \Gamma\left(\frac{N+1}{2}\right)^2
   \pi}.
\end{align*}

First, we estimate $q$ from below. Using Gautschi's inequality for the Gamma function,
\[
\frac{\Gamma\left( \frac{N}{2} + 1 \right)}{\Gamma\left( \frac{N}{2} + \frac{1}{2} \right)}>\left( \frac{N}{2} + \frac{1}{2} \right)^{\frac{1}{2}},
\]
and then
\begin{align*}
\frac{\Gamma\left(\frac{N}{2}+3\right)\Gamma \left(\frac{N+2}{2}\right)}{\Gamma\left(\frac{N+1}{2}\right)^2 }
=\frac{\left(\frac{N}{2}+2\right)\left(\frac{N}{2}+1\right)\Gamma\left(\frac{N}{2}+1\right)^2}{\Gamma\left(\frac{N+1}{2}\right)^2 }
>\left(\frac{N+4}{2}\right)\left(\frac{N+2}{2}\right)\left( \frac{N+1}{2}\right).
\end{align*}
Then,
\begin{align*}
q(N)>\frac{
  (N-4) (N-2) N (N+3)^2 (N+4)}{(N+1)^4} \frac{16^{\frac{N+1}{N}}}{2 \pi}\left( \frac{N+1}{2}\right)
  >4\frac{N(N-4)(N+1)}{\pi},
\end{align*}
where we used that $\frac{(N-2)(N+3)^2(N+4)}{(N+1)^4}>1$ and $\frac{16^{\frac{N+1}{N}}}{4}>4$ for $N\geq 5$. Next, we estimate $-p(N)$ from above,
\begin{align*}
-p(N)
&=-\left(4^{\frac{1}{N}}-1\right)^2 N^3-3 \left(16^{\frac{1}{N}}-1\right)
   N^2-\left(2^{\frac{2}{N}+5}-5\ 2^{\frac{N+4}{N}}-34\right) N+24\\
&<\left(5\ 2^{1+\frac{4}{N}}+34\right) N+24
<\left((5)(4)+34\right) N+24=54N+34
\end{align*}
for $N\geq 5$.  Then, for $N\geq 9,$
\begin{align*}
-p(N)<54N+34<4\frac{N(N-4)(N+1)}{\pi}<q(N),
\end{align*}
which yields the symmetry breaking result for all $N\geq 9.$ For $N = 3,\ldots,8,$ we compute the corresponding values of $h_1(N)-h_2(N)$ directly in Table~\ref{tab1}.

\begin{table}
\begin{center}
{\small
\begin{tabular}{||c c c c||} 
 \hline
 $N$ & $m_{rad}=h_1(N)$ & $h_2(N)$ & $h_1(N)-h_2(N)$ \\ [0.5ex] 
 \hline\hline
 3 & -0.00272396 & -0.0795216 & 0.0767976 \\
 4 & -0.00156672 & -0.0619457 & 0.060379 \\
 5 & -0.000908657 & -0.0466251 & 0.0457165 \\
 6 & -0.000525701 & -0.0339151 & 0.0333894 \\
 7 & -0.000301342 & -0.0238506 & 0.0235493 \\
 8 & -0.000170448 & -0.0162296 & 0.0160591 \\
[1ex] 
 \hline
\end{tabular}
}
\end{center}
\caption{In the second column the radial least energy level $m_{rad}$ (see \eqref{mrad:def}) is computed. The last column shows the difference $h_1(N)-h_2(N)$, which is always positive, and implies the symmetry breaking for dimensions $N=3,\ldots,26$.
}\label{tab1}
\end{table}
\end{proof}

\begin{proof}[Proof of Corollary \ref{coro:intro}]
For $q\in (1-\eps,1+\eps)$, the result follows from Theorem \ref{prop:convergence direct} and Theorem \ref{symm:break:thm}. For $q\in (0,\eps)$, the result follows from Theorem \ref{p, q 0} and \cite[Theorem 1.2]{PW}.
\end{proof}

\section{Multiplicity result: proof of Theorem~\ref{prop:mult}}\label{sec:mul}

Consider the functional $\Phi(f, g):=- \int_\Omega f Kg$, and take $S:=\{ (f, g) \in X: \gamma_1 \norm{f}_\alpha^\alpha+ \gamma_2 \norm{g}_\beta^\beta=1\}$.  
Notice that $\Phi \in C^1(S)$ is even and bounded below on $S$, as
\[ \inf_S \Phi = - \sup_S \int_\Omega f K g = - D_{p, q} > -\infty.\] 
The last inequality is a consequence of Sobolev embeddings and Young inequality, see \cite[p.755]{PST}.

We will apply \cite[Corollary 4.1]{Szulkin} and, as a first step, we show the following

\begin{lemma}\label{lem:PS}
The functional $\Phi$ satisfies the Palais Smale (PS) condition on $S$ at every level $c\neq 0$. 
\end{lemma}
\begin{proof}
Let us take a (PS) sequence at level $c$ on $S$, namely let $(f_n, g_n) \in S$ such that 
\[ \Phi(f_n, g_n) \to c, \quad \Phi|'_{S} (f_n, g_n) \to 0, \]
for some $c\neq 0$. In particular,
\begin{equation}\label{PS1}
\int_\Omega f_n Kg_n \to c\neq 0,
\end{equation}
and there exists $\lambda_n$ such that, for any $(\varphi, \psi)\in X$, 
\begin{equation}\label{PS}  - \int_\Omega f_n K \psi - \int_\Omega g_n K \varphi - \lambda_n \int_\Omega |f_n|^{\frac 1 p -1} f_n \varphi - \lambda_n \int_\Omega |g_n|^{\frac 1 q -1} g_n \psi= o_n(1)\norm{(\varphi, \psi)}_X. 
\end{equation}

Take $\varphi= \gamma_1f_n$ and $\psi=\gamma_2 g_n$ in \eqref{PS}. Then
\begin{equation}\label{PS3} - \int_\Omega f_n K g_n - \lambda_n   = o_n(1)\norm{(f_n, g_n)}_X). 
\end{equation}
Since $(f_n,g_n)\in S$, then $(f_n,g_n)$ is bounded in $X$ and, from \eqref{PS1} and \eqref{PS3}, $\lambda_n$ is a bounded sequence, converging (up to a subsequence) to some $\bar \lambda$. Moreover,  $\bar \lambda\neq 0$, since $c\neq 0$.

Let $(f,g)\in X$ be such that $(f_n,g_n)\rightharpoonup (f,g)$ in $X$. Now take $\psi=0$ and $\varphi=f_n-f$ in \eqref{PS}. One has
\[ -\int_\Omega g_n K(f_n- f) - \lambda_n  \int_\Omega |f_n|^{\frac 1p -1} f_n(f_n- f)= o_n(1)\norm{f_n -f}_\alpha, \]
whence (by the compactness of the operator $K$),
\begin{align*} \lim_n \lambda_n  \int_\Omega |f_n|^{\alpha} &= \bar \lambda  \int_\Omega |f|^\alpha - \lim_n \int_\Omega g_n K(f_n -f) + o(\norm{f_n -f}_\alpha) \\
&=\bar \lambda \int_\Omega |f|^\alpha + o(\norm{f_n -f}_\alpha), \end{align*}
and $\norm{f_n}_\alpha \to \norm{f}_\alpha$, since $\bar \lambda \ne 0$. Doing an analogous reasoning with $\psi=g_n-g$ and $\varphi=0$, we deduce that $(f_n,g_n)\to (f,g)$ strongly in $X$, concluding the proof.
\end{proof}

\begin{proof}[Proof of Theorem~\ref{prop:mult}]
Let us 
define
\[ c_j:= \inf_{A \in \Gamma_j} \sup_{(f, g) \in A} \Phi(f, g), \]
where 
\[ \Gamma_j:=\{ A \subset S: A \text{ closed, symmetric and compact, } \gamma(A) \ge j \}, \]
and $\gamma(A)$ denotes the genus of $A$, see for instance \cite[Definition 5.1]{Struwe}. 
Let us show that $\Gamma_j \ne \emptyset$ for any $j$, and that $c_j \ne 0$.

For any $k \ge 1$, take 
\[ 
A_k:=\{ (f, f) \in S: f= \sum_{i=1}^k a_i \varphi_i,  a_i\in \R \} 
\] 
 where $\varphi_i$ is the $i$-th eigenfunction of the Laplace operator with Neumann boundary conditions with coresponding eigenvalue $\mu_i$; in particular, 
 \[ \int_\Omega \varphi_i K \varphi_i= \frac{1}{\mu_i} \int_\Omega \varphi_i^2\quad \forall i,\qquad \text{ and } \int_\Omega \varphi_i K \varphi_j=\int_\Omega \varphi_i\varphi_j=0\quad \forall i\neq j.
 \]
 Then $\gamma(A_k) =k$ (by \cite[Proposition 5.2]{Struwe}), 
 and 
 \begin{align*} 
 \Phi(f,f)
 &=- \int_\Omega f Kf= - \sum_{i=1}^k a_i^2 \int_\Omega \varphi_i K \varphi_i=  - \sum_{i=1}^k a_i^2  \frac{1}{\mu_i} \int_\Omega \varphi_i^2 \le - \frac{1}{\mu_k} \norm{f}_2^2\\
 &\le -\frac{1}{\mu_k} C(\gamma_1 \norm{f}_\alpha^\alpha+ \gamma_2 \norm{g}_\beta^\beta)=-\frac C{\mu_k}<0 
 \end{align*}
 for any $(f, f) \in A_k$, since norms are equivalent in finite dimension spaces.

 By  \cite[Corollary 4.1]{Szulkin} and Lemma~\ref{lem:PS}, it follows that there exist infinitely many critical points $(f_k, g_k)$ restricted to $S$ at $c_k$ for any $k \in \N$. Thus, by the Lagrange Multipliers Theorem, there exist a sequence  $\lambda_k$ such that 
\[ \int_\Omega \psi K f_k=\lambda_k  \int_\Omega |g_k|^{\frac 1 q-1} g_k \psi, \quad \int_\Omega \varphi K g_k  =\lambda_k \int_\Omega |f_k|^{\frac 1 p-1} f_k \psi, \]
for $(\varphi, \psi) \in X$. This implies (see also \cite[Proposition 2.5]{PST}) that 
\begin{equation}\label{eq:fg} K_p g_k= \lambda_k   |f_k|^{\frac 1 p-1} f_k, \qquad  K_q f_k= \lambda_k  |g_k|^{\frac 1 q-1} g_k, \end{equation}
and (since $\gamma_1+\gamma_2=1$) we have $c_k=\lambda_k $. Since $pq \ne 1$, 
\begin{equation}\label{scaling} (\tilde f_k, \tilde g_k):=((\lambda_k )^s f_k, (\lambda_k)^t g_k), \quad \text{ with } s=-p \frac{q+1}{pq-1}, \, t=-q \frac{p+1}{pq-1} \end{equation}
satisfy 
\[ K_p \tilde g_k = |\tilde f_k|^{\frac 1 p-1} \tilde f_k, \quad K_q \tilde f_k = |\tilde g_k|^{\frac 1 q-1} \tilde g_k.  \]
Thus, we find infinitely  many distinct strong solutions to 
\eqref{system}.

Let us  justify why they are distinct. Since $c_k=\lambda_k$ and recalling \eqref{scaling}, we just need to consider the case in which for some $k \ne j$ one has $c_k \ne c_j$, $\tilde f_k = \tilde f_j$ and $\tilde g_k=\tilde g_j$. 
Then 
\[ f_k=\lambda_j^s \lambda_k^{-s} f_j, \quad g_k=\lambda_j^t \lambda_k^{-t} g_j.  \]
Therefore, since $c_k=\lambda_k$, 
\[ c_k= - \int_\Omega f_k K g_k = - \lambda_j^{s+t} \lambda_k^{-(s+t)} \int_\Omega f_j K g_j= \lambda_j^{s+t} \lambda_k^{-(s+t)} c_j = (-c_k)^{-(s+t)} (-c_j)^{s+t} c_j,  \]
which would imply $c_k=c_j$, a contradiction.
\end{proof}

\section*{Acknowledgments.}
Delia Schiera and Hugo Tavares are partially supported by the Portuguese government through FCT - Funda\c c\~ao para a Ci\^encia e a Tecnologia, I.P., under the projects UIDB/04459/2020 with DOI identifier 10-54499/UIDP/04459/2020 (CAMGSD), and PTDC/MAT-PUR/1788/2020 with DOI identifier 10.54499/PTDC/MAT-PUR/1788/2020 (project NoDES). Delia Schiera is also partially supported by FCT, I.P. and by COMPETE 2020 FEDER funds, under the Scientific Employment Stimulus - Individual Call (CEEC Individual), DOI identifier 10.54499/2020.02540.CEECIND/CP1587/CT0008.  A. Saldaña is supported by CONAHCYT grant CBF2023-2024-116 (Mexico) and by UNAM-DGAPA-PAPIIT grant IN102925 (Mexico).

\makeatletter
\newcommand{\adjustmybblparameters}{\setlength{\itemsep}{2pt}\setlength{\parsep}{0pt}}
\let\ORIGINALlatex@openbib@code=\@openbib@code
\renewcommand{\@openbib@code}{\ORIGINALlatex@openbib@code\adjustmybblparameters}
\makeatother

\bibliographystyle{plain}
\bibliography{SST.bib}

\vspace{.3cm}

\noindent \textbf{Alberto Salda\~na}\\
Instituto de Matemáticas\\
Universidad Nacional Autónoma de México\\
Circuito Exterior, Ciudad Universitaria\\
04510 Coyoacán, Ciudad de México, Mexico\\
\texttt{alberto.saldana@im.unam.mx}
\vspace{.3cm}

\noindent \textbf{Delia Schiera, Hugo Tavares}\\
CAMGSD - Centro de An\'alise Matem\'atica, Geometria e Sistemas Din\^amicos\\
Departamento de Matemática do Instituto Superior Técnico\\
Universidade de Lisboa\\
Av. Rovisco Pais\\
1049-001 Lisboa, Portugal\\
\texttt{delia.schiera@tecnico.ulisboa.pt, hugo.n.tavares@tecnico.ulisboa.pt}

\end{document}